\documentclass[a4paper,reqno,oneside,11pt]{amsart}

\usepackage{lineno}

\calclayout

\usepackage{amsmath}
\usepackage{amsfonts}
\usepackage{amssymb}
\usepackage{verbatim,enumitem,graphicx}
\usepackage[colorlinks]{hyperref}
\usepackage{dsfont}
\usepackage{tikz}

\usepackage[utf8]{inputenc}

\usepackage{enumitem}
\usepackage{MnSymbol}

\setenumerate{leftmargin=*} 

\newtheorem{theorem}{Theorem}[section]
\newtheorem{lemma}[theorem]{Lemma}
\newtheorem{proposition}[theorem]{Proposition}

\newtheorem{definition}[theorem]{Definition}

\newtheorem{problem}[theorem]{Problem}
\newtheorem{claim}[theorem]{Claim}

\newcommand{\oldqed}{}
\def\endofClaim{\hfill\scalebox{.6}{$\Box$}}

\newenvironment{claimproof}[1][Proof]{
\renewcommand{\oldqed}{\qedsymbol}
\renewcommand{\qedsymbol}{\endofClaim}
\begin{proof}[#1]
}{
\end{proof}
\renewcommand{\qedsymbol}{\oldqed}
}

\newcommand{\By}[2]{\overset{\mbox{\tiny{#1}}}{#2}}

\newcommand{\leBy}[1]{    \By{#1}{\le} }

\newcommand{\Gnp}{G(n,p)}

\newcommand{\cH}{\mathcal{H}}
\newcommand{\cM}{\mathcal{M}}

\newcommand{\cJ}{\mathcal{J}}

\newcommand{\cK}{\mathcal{K}}
\newcommand{\cI}{\mathcal{I}}
\newcommand{\cT}{\mathcal{T}}

\newcommand{\PP}{\mathbb{P}}
\newcommand{\EE}{\mathbb{E}}

\newcommand{\eps}{\varepsilon}

\renewcommand{\subset}{\subseteq}

\newcommand*\patchAmsMathEnvironmentForLineno[1]{%
\expandafter\let\csname old#1\expandafter\endcsname\csname #1\endcsname
\expandafter\let\csname oldend#1\expandafter\endcsname\csname end#1\endcsname
\renewenvironment{#1}%
{\linenomath\csname old#1\endcsname}%
{\csname oldend#1\endcsname\endlinenomath}}%
\newcommand*\patchBothAmsMathEnvironmentsForLineno[1]{%
\patchAmsMathEnvironmentForLineno{#1}%
\patchAmsMathEnvironmentForLineno{#1*}}%
\AtBeginDocument{%
\patchBothAmsMathEnvironmentsForLineno{equation}%
\patchBothAmsMathEnvironmentsForLineno{align}%
\patchBothAmsMathEnvironmentsForLineno{flalign}%
\patchBothAmsMathEnvironmentsForLineno{alignat}%
\patchBothAmsMathEnvironmentsForLineno{gather}%
\patchBothAmsMathEnvironmentsForLineno{multline}%
}

\begin{document}

\title{Triangles in randomly perturbed graphs}

\author[J.~B\"{o}ttcher]{Julia B\"{o}ttcher$^\ast$}
\author[O.~Parczyk]{Olaf Parczyk$^{\dag}$}
\author[A.~Sgueglia]{Amedeo Sgueglia$^\ast$}
\author[J.~Skokan]{Jozef Skokan$^{\ast,\S}$}

\thanks{$^\ast$ Department of Mathematics, London School of Economics, London, WC2A 2AE, UK. \\
	\textit{E-mail:} \{\texttt{j.boettcher|a.sgueglia|j.skokan}\}\texttt{@lse.ac.uk}}
\thanks{$^\S$ Department of Mathematics, University of Illinois at Urbana-Champaign,  Urbana, IL 61801, USA}
\thanks{$^\dag$ 
Institute of Mathematics, Freie Universität Berlin, Arnimallee 3, 14195 Berlin, Germany. \\
This research was conducted while OP was a visiting fellow at the London School of Economics and supported by the Deutsche Forschungsgemeinschaft (DFG, Grant PA 3513/1-1).\\
\textit{E-mail:} \texttt{parczyk@mi.fu-berlin.de}}

\begin{abstract}
	We study the problem of finding pairwise vertex-disjoint triangles in the
	randomly perturbed graph model, which is the union of any $n$-vertex graph $G$
	satisfying a given minimum degree condition and the binomial random graph $G(n,p)$. We prove
	that asymptotically almost surely $G \cup G(n,p)$ contains at least $\min
	\{\delta(G), \lfloor n/3 \rfloor \}$ pairwise vertex-disjoint triangles,
	provided $p \ge C \log n/n$, where $C$ is a large enough constant. This is a
	perturbed version of an old result of Dirac.
	
	Our result is asymptotically
	optimal and answers a question of Han, Morris, and Treglown~[RSA, 2021, no.~3, 480--516] in a
	strong form. We also prove a stability
	version of our result, which in the case of pairwise vertex-disjoint triangles
	extends a result of Han, Morris, and Treglown~[RSA, 2021, no.~3, 480--516].
	Together with a result of Balogh, Treglown, and
	Wagner~[CPC, 2019, no.~2, 159--176] this fully resolves the existence of
	triangle factors in randomly perturbed graphs.
	
	We believe that the methods introduced in this paper are useful for a variety
	of related problems:
	we discuss possible generalisations to 
	clique factors, cycle factors, and $2$-universality.
\end{abstract}

\maketitle

\section{Introduction and main result}

The problem of subgraph containment is one of the most studied questions in extremal and probabilistic combinatorics.
In this paper we are interested in conditions that guarantee the containment of pairwise vertex-disjoint triangles in different graph models.
In particular when a graph $G$ on $v(G)$ vertices contains $\lfloor v(G)/3 \rfloor$ pairwise vertex-disjoint copies of triangles, we say that $G$ contains a triangle factor.
We will consider minimum degree conditions in dense graphs, lower bounds on the edge-probability in random graphs, and a combination of both.

From Mantel's theorem on the maximum number of edges in a triangle-free graph it follows that any $n$-vertex graph with minimum degree larger than $n/2$ contains a triangle.
The first result on the containment of a triangle factor is due to Corrádi and Hajnal~\cite{corradi1963maximal}, who proved that any $n$-vertex graph $G$ with minimum degree $\delta(G) \ge 2n/3$ contains a triangle factor.
From this it is not hard to derive a more general result for graphs with smaller minimum degree, which was first proved by Dirac~\cite{dirac1963maximal}.

\begin{theorem}[Dirac~\cite{dirac1963maximal}]
	\label{thm:dirac}
	Any $n$-vertex graph $G$ with $n/2 \le \delta(G) \le 2n/3$ contains at least $2\delta(G)-n$ pairwise vertex-disjoint triangles.
\end{theorem}

Given an integer $m$ with $n/2 \le m \le 2n/3$, the tripartite complete graph with parts of size $2m-n$, $n-m$, and $n-m$ shows the result is best possible.
Moreover, a stability version of Theorem~\ref{thm:dirac} was proved by Hladk\'{y}, Hu, and Piguet~\cite{hladky2019komlos}.

Sparser graphs do not necessarily contain triangles, but most of them do if there are enough edges.
To determine a cut-off point for the containment of a subgraph in the \emph{binomial random graph} $G(n,p)$ we define a threshold function.
Before stating its definition, we remark that we say that a property $\mathcal{A}$ holds asymptotically almost surely (a.a.s.) in the random graph $G(n,p)$ if $\lim_{n \to \infty}\PP[G(n,p) \in \mathcal{A}]=1$.
Given a graph $H$, the function $\hat{p} : \mathbb{N} \rightarrow [0,1]$ is called the \emph{threshold} for the containment of $H$ in $G(n,p)$ if a.a.s. $H \subseteq G(n,p)$ for $p = \omega (\hat{p})$ and a.a.s. $H \not\subseteq G(n,p)$ for $p=o(\hat{p})$.
Already in one of the early papers on random graphs by Erd\H{o}s and R\'enyi~\cite{erdHos1960evolution} from $1960$ the threshold for a single triangle was determined as $1/n$.
However, the problem for a triangle factor is much harder and was eventually solved by Johannson, Kahn, and Vu~\cite{johansson2008factors} in $2008$, as part of a more general result, and the threshold was located at $n^{-2/3} \log^{1/3}n$, with an even sharper transition than in our definition of threshold.
When one requires $\eps n$ to $(1-\eps)n$ pairwise vertex-disjoint triangles, the problem is easier and the threshold is $n^{-2/3}$ as proved by Ruci\'nski~\cite{rucinski1992matching} in $1992$.
Note that the the spanning version requires an extra $\log^{1/3}n$, and this logarithmic term is essential to ensure that a.a.s.~every vertex is contained in a triangle.

Bohman, Frieze, and Martin~\cite{bohman2003} combined the random graph model and the deterministic minimum-degree model by asking how many random edges one needs to add to a dense graph with small linear minimum degree, such that it contains a Hamilton cycle.
More precisely, they introduced the model of \emph{randomly perturbed graphs} as the union of an $n$-vertex graph $G_\alpha$ with minimum degree at least $\alpha n$ and the random graph $G(n,p)$.
Given $H$ and fixed $\alpha$, we are then interested in lower bounds on the edge-probability $p$ such that a.a.s.~$G_\alpha \cup G(n,p)$ contains $H$ for any $G_\alpha$.
A lower bound $\hat{p}$ is optimal when in addition there exists a $G_\alpha$ for which a.a.s.~$G_\alpha \cup G(n,p)$ does not contain $H$ if $p=o(\hat{p})$, in which case we call $\hat{p}$ the threshold for the containment of $H$ in the randomly perturbed model.
Note that this threshold $\hat{p}$ only depends on $\alpha$ and $H$.
In recent years there has been a lot of work on embeddings of spanning graphs in randomly perturbed graphs.
Most results in this model focus on the extreme cases with small $\alpha>0$~\cite{balogh2019tilings, bohman2004adding, bottcher2017embedding, joos2020spanning, krivelevich2016cycles, krivelevich2017bounded, krivelevich2006smoothed, mcdowell2018hamilton} or small $p$~\cite{antoniuk2020high, bedenknecht2019powers, dudek2020powers, nenadov2018sprinkling}.
More recently, Han, Morris, and Treglown~\cite{han2019tilings} started a more thorough investigation of the intermediate regime.
The goal is to determine the perturbed threshold for every $\alpha$ from $\alpha=0$, where we can rely only on $G(n,p)$, to the $\alpha$ where the structure already exists in $G_\alpha$ alone and $p=0$ is sufficient.

It is easy to see that with $0<\alpha \le 1/2$, $G_\alpha \cup G(n,p)$ a.a.s.~contains a triangle when $p \ge C/n^2$ and $C$ is a sufficiently large constant depending on $\alpha$.
This is asymptotically optimal, as $G(n,p)$ with $p = o(n^2)$ a.a.s.~is empty.
Together with the cases for $\alpha=0$ and $\alpha > 1/2$ discussed above, this completes all the range of $\alpha$ for the containment of a single triangle. 
For a triangle factor we already know the threshold when $\alpha=0$ from the random graph model, and when $\alpha \ge 2/3$ we do not need random edges at all.
For any $\alpha>0$, Balogh, Treglown, and Wagner~\cite{balogh2019tilings} showed that $p \ge C n^{-2/3}$ is always sufficient (with $C$ depending on $\alpha$).
This is asymptotically optimal in the case $0<\alpha<1/3$, as with $G_\alpha$ the complete bipartite graph with classes of size $\alpha n$ and $(1-\alpha n)$, we need a linear number of triangles with all edges from the random graph, for which $n^{-2/3}$ is the threshold as discussed above.
For $1/3 < \alpha <2/3$, Han, Morris, and Treglown~\cite{han2019tilings} proved that then already $p \ge C/n$ is enough for a triangle factor.
Again this is asymptotically optimal as, when $G_\alpha$ is the complete tripartite graph with classes of size $\alpha n/2$, $\alpha n/2$ and $(1-\alpha)n$, we need a linear number of edges from  $G(n,p)$.
Together these results can be summarised as follows.

\begin{theorem}[Balogh, Treglown, and Wagner~\cite{balogh2019tilings}  and Han, Morris, and Treglown~\cite{han2019tilings}]
	\label{thm:k3_factor}
	Given any $\alpha \in (0,1/3) \cup (1/3,2/3)$ there exists $C>0$ such that the following holds. 
	For any $n$-vertex graph $G$ with minimum degree $\delta(G) \ge \alpha n$, a.a.s.~there is a triangle factor in~$G \cup G(n,p)$ provided that $p \ge C n^{-2/3}$ if $\alpha \in (0,1/3)$ and $p \ge C n^{-1}$ if $\alpha \in (1/3,2/3)$.
\end{theorem}

In this paper we close the remaining open case for the triangle factor, that is
$\alpha = 1/3$, for which we show that $p \ge C \log n/n$ is sufficient.

\begin{theorem}
	\label{thm:zero}
	There exists $C>0$ such that for any $n$-vertex graph $G$ with minimum degree $\delta(G) \ge n/3$, we can a.a.s.~find a triangle factor in $G \cup G(n,p)$, provided that $p\ge C \log n /n$. 
\end{theorem}

To see
that the bound on $p$ in Theorem~\ref{thm:zero} is asymptotically optimal
consider the complete bipartite graph $G=K_{n/3,2n/3}$ and denote the partition
classes by $A$ and $B$ with $|A|<|B|$. By Markov's inequality and with $p \le
\tfrac 12 \log n/n$ a.a.s.~there are $O(\log^4 n)$ triangles within $B$ and
a.a.s.~there is a polynomial number of vertices in the class $B$ without any
neighbours in $B$~\cite[Theorem~6.36]{JLR}. However, for a triangle factor to
exist, for each triangle with at most one vertex in $B$, there must be at least
one triangle fully contained in $B$. In conclusion, a.a.s.~$G \cup G(n,p)$ does
not contain a triangle factor and the $\log n$-term is needed for local reasons
similarly as discussed above for the triangle factor in $G(n,p)$.

This result closes the problem of determining, given $\alpha \in [0,1]$, the
threshold for a triangle factor in $G_\alpha \cup G(n,p)$ and we refer to
Table~\ref{fig:K3-factor} for a summary. Note that the threshold is within a
constant factor in the intervals $(0,1/3)$ and $(1/3,2/3)$, while it jumps at
$\alpha=0$, $1/3$, and $2/3$.
\begin{table}[htbp]
	\begin{center}
		\begin{tabular}{||c||c c c c c||} 
			\hline
			$\alpha$ & $\alpha=0$ & $0<\alpha<1/3$ & $\alpha=1/3$ & $1/3<\alpha<2/3$ & $2/3 \le \alpha$\\
			\hline\hline
			$p$ & $n^{-2/3} \log^{1/3} n$ & $n^{-2/3}$ & $n^{-1} \log n$ & $n^{-1}$ & $0$ \\ 
			\hline
		\end{tabular}	
	\end{center}
	\caption{\label{fig:K3-factor}Triangle factor containment in $G_\alpha \cup G(n,p)$, where $\delta(G_\alpha) \ge \alpha n$.}
\end{table}

Theorem~\ref{thm:zero} is a special case of the following theorem, which is our
main result. It answers the question which minimum degree condition is needed
in the randomly perturbed graph model with $p= C\log n/n$ to enforce $k$
vertex-disjoint triangles for any~$1\le k\le \lfloor n/3\rfloor$.

\begin{theorem}[Main result]
	\label{thm:main}
	There exists $C>0$ such that for any $n$-vertex graph $G$ we can a.a.s.~find at least $\min\{ \delta(G), \lfloor n/3 \rfloor \}$ pairwise vertex-disjoint triangles in $G \cup G(n,p)$, provided that $p\ge C \log n /n$. 
\end{theorem}

This is a perturbed version of the result by Dirac on vertex-disjoint triangles in dense graphs (Theorem~\ref{thm:dirac}).
We are not aware of other results in the randomly perturbed graph model that
consider large but not spanning structures.

Theorem~\ref{thm:main} is basically optimal in terms of the number of triangles, because given $1 \le m < n/3$, then $G=K_{m,n-m}$ has minimum degree $\delta(G)=m$, and there can be at most $m$ pairwise vertex-disjoint triangles using each at least one edge of $G$, and at most $O(\log^4 n)$ additional triangles solely coming from $G(n,p)$.
The bound on $p$ is asymptotically optimal as it is in Theorem~\ref{thm:zero}, but we remark that when $m$ is `significantly smaller' than $n/3$, then already $p \ge C / n$ is sufficient to a.a.s.~find $m$ pairwise vertex-disjoint triangles in $K_{m,n-m} \cup G(n,p)$.
We call $K_{m,n-m}$ the extremal graph.
See the concluding remarks (Section~\ref{sec:concluding}) for more details.

In addition, we prove a stability version (Theorem~\ref{thm:non-extremal}) of
our main result, which allows us to work with edge probability only $p=C/n$ for
graphs $G$ that are not `close'\footnote{Basically this is in edit distance, but
	in addition we require that no vertex can have a degree in $G$ that is much
	smaller than the corresponding degree in the extremal graph. A precise
	condition is given in Definition~\ref{def:stability}.} to the extremal graph.

We believe that the methods we introduce for proving our results are valuable
for other questions concerning randomly perturbed graphs; we discuss some possible
directions and open problems in Section~\ref{sec:concluding}.
One important novel ingredient in our proofs is that we can find a triangle
factor in a graph on three vertex sets $U, V, W$ of the same size, where $(V,U)$
and $(V,W)$ are super-regular pairs and between $U$ and $W$ we have random edges
with probability $p \ge C \log n/n$ (see Lemma~\ref{lem:tripartite}). 

\subsection*{Organisation}
The rest of this paper is organised as follows.
In Section~\ref{sec:main} we state our Stability Theorem
(Theorem~\ref{thm:non-extremal}), one result (Theorem~\ref{thm:extremal}) that
deals with the `extremal' case of Theorem~\ref{thm:main}, and one result
(Theorem~\ref{thm:sublinear}) that deals with the case of small minimum degrees; we shall show that these
three theorems together imply our main result (Theorem~\ref{thm:main}).

In Section~\ref{sec:tools} we then introduce some tools that we will use later.
In Section~\ref{sec:overview}, we outline the proofs of
Theorems~\ref{thm:non-extremal},~\ref{thm:extremal}, and~\ref{thm:sublinear} and
we state the auxiliary lemmas we use in their proofs.
In Section~\ref{sec:extremal} we prove Theorem~\ref{thm:extremal}, in
Section~\ref{sec:non-extremal} we prove Theorem~\ref{thm:non-extremal}, and in
Section~\ref{sec:fewtriangles} we
prove Theorem~\ref{thm:sublinear}.
The auxiliary lemmas are proved in Section~\ref{sec:auxiliary}.

Finally we give concluding remarks and pose some open questions in
Section~\ref{sec:concluding}. A few supplementary proofs are moved to
Appendix~\ref{sec:appendix}.

\subsection*{Notation}
For numbers $a$, $b$, $c$, we write $a = b \pm c$ for $b-c \le a \le b+c$.
Moreover, for non-negative $a$,$b$ we write $0<a \ll b$, when we require $a \le f(b)$ for some function $f \colon \mathbb{R}_{>0} \mapsto \mathbb{R}_{>0}$.
We will only use this to improve readability and in addition to the precise dependencies of the constants.

We use standard graph theory notation.
For a graph $G$ on vertex set $V$ and two disjoint sets $A$, $B \subset V$, let $G[A]$ be the subgraph of $G$ induced by $A$, $G[A,B]$ be the bipartite subgraph of $G$ induced by sets $A$ and $B$, $e(A)$ be the number of edges with both endpoints in $A$ and $e(A,B)$ be the number of edges with one endpoint in $A$ and the other one in $B$.
We will also use standard Landau notation for $f,g : \mathbb{N} \rightarrow \mathbb{R}_{>0} : f = o(g)$ if and only if $\lim_{n \to \infty}f(n)/g(n)=0$ and $f = \omega(g)$ if and only if $g = o(f)$.

\section{Stability version and proof of the main result}
\label{sec:main}

We already discussed how the probability $p \ge C \log n/n$ can not be significantly lowered in Theorem~\ref{thm:main}.
However, we are able to show that when the minimum degree of $G$ is linear in $n$, then with $m=\min \{\delta(G),n/3\}$, the complete bipartite graph $K_{m,n-m}$ is the unique extremal graph for Theorem~\ref{thm:main}, in the sense that if the graph $G$ is not `close' to $K_{m,n-m}$ then a.a.s.~$G \cup G(n,p)$ contains $m$ pairwise vertex-disjoint triangles already at probability $p \ge C/n$ and we can even assume a slightly smaller minimum degree on $G$.
To formalise this we introduce the following notion of stability for an $n$-vertex graph $G$.

\begin{definition}[$(\alpha,\beta)$-stable]
	\label{def:stability}
	For $0 < \beta < \alpha < 1/2$, we say that an $n$-vertex graph $G$ is \emph{$(\alpha,\beta)$-stable} if there exists a partition of $V(G)$ into two sets $A$ and $B$ of size $|A|=(\alpha \pm \beta) n$ and $|B|=(1-\alpha \pm \beta) n$ such that the minimum degree of the bipartite subgraph $G[A,B]$ of $G$ induced by $A$ and $B$ is at least $\alpha n/4$, all but at most $\beta n$ vertices from $A$ have degree at least $|B|-\beta n$ into $B$, all but at most $\beta n$ vertices from $B$ have degree at least $|A|-\beta n$ into $A$, and $G[B]$ contains at most $\beta n^2$ edges.
\end{definition}

The stability condition with $\alpha=1/3$ says that the size of $B$ is roughly double the size of $A$, there is a minimum degree condition between $A$ and $B$, in each part all but at most a few vertices see
most of the other part, and the set $B$ is almost independent.
Note that for $0 < \alpha \le 1/3$ and $m=\alpha n$ an integer, the complete bipartite graph $K_{m,n-m}$ is $(\alpha,\beta)$-stable with $\beta=0$.
We prove the following stability result in Section~\ref{sec:non-extremal}.

\begin{theorem}[Stability Theorem]
	\label{thm:non-extremal}
	For $0 < \beta < 1/12$ there exist $\gamma>0$ and $C>0$ such that for any $\alpha$ with $4 \beta \le \alpha \le 1/3$ the following holds.
	Let $G$ be an $n$-vertex graph with minimum degree $\delta(G) \ge \left(\alpha - \gamma \right) n$ that is not $(\alpha,\beta)$-stable.
	With $p\ge C /n$ a.a.s.~the perturbed graph $G \cup G(n,p)$ contains at least $ \min \{ \alpha n, \lfloor n/3 \rfloor \} $ pairwise vertex-disjoint triangles.
\end{theorem}

The result is best possible as $G$ can be bipartite and have no triangles, in which case we need at least a linear number of edges from the random graph to find a linear number of pairwise vertex-disjoint triangles in $G \cup G(n,p)$. On the other hand, the logarithmic factor is needed for the extremal graph. When the graph $G$ is $(\alpha, \beta)$-stable for a small enough $\beta>0$ then we prove the following in Section~\ref{sec:extremal}.

\begin{theorem}[Extremal Theorem]
	\label{thm:extremal}
	For $0 < \alpha_0 \le 1/3$ there exist $\beta,\gamma>0$ and $C>0$  such that for any $\alpha$ with $\alpha_0 \le \alpha \le 1/3$ the following holds.
	Let $G$ be an $n$-vertex graph with minimum degree $\delta(G) \ge \left( \alpha -\gamma \right) n$ that is $(\alpha,\beta)$-stable.
	With $p\ge C \log n /n$ a.a.s.~the perturbed graph $G \cup G(n,p)$ contains at least $\min\{ \delta(G), \lfloor \alpha n \rfloor  \}$ pairwise vertex-disjoint triangles.
\end{theorem}

Indeed our argument will give slightly more.
If $G$ is $(\alpha,\beta)$-stable and $|A| \ge \alpha n$, then we can a.a.s.~find $\lceil \alpha n \rceil$ if $\alpha < 1/3$ and $\lfloor n/3 \rfloor$ if $\alpha=1/3$ pairwise vertex-disjoint triangles in $G \cup G(n,p)$ (even when the minimum degree in $G$ is smaller than $\alpha n$).
Also, although as discussed above the $\log n$-factor can not be avoided in general, when $|A|-\alpha n$ is linear in $n$ our proof does not need such a $\log n$-factor.

The proofs of Theorem~\ref{thm:non-extremal} and~\ref{thm:extremal} use regularity.
When the minimum degree gets smaller we can avoid the transition to sparse regularity and prove the following in Section~\ref{sec:fewtriangles} with a more elementary argument.

\begin{theorem}[Sublinear Theorem]
	\label{thm:sublinear}
	There exists $C>0$ such that the following holds for any $1 \le m \le n/256$ and any $n$-vertex graph $G$ of minimum degree $\delta(G) \ge m$.
	With $p \ge C \log n/n$ a.a.s.~the perturbed graph $G \cup G(n,p)$ contains at least $m$ pairwise vertex-disjoint triangles.   
\end{theorem}
Theorem~\ref{thm:main} easily follows from Theorems~\ref{thm:non-extremal},~\ref{thm:extremal}, and~\ref{thm:sublinear}.

\begin{proof}[Proof of Theorem~\ref{thm:main}]
	Let $\beta_{\ref{thm:extremal}}, \gamma_{\ref{thm:extremal}}>0$ and $C_{\ref{thm:extremal}}$ be given by Theorem~\ref{thm:extremal} on input $\alpha_0=1/256$.
	Then let $\gamma_{\ref{thm:non-extremal}}>0$ and $C_{\ref{thm:non-extremal}}$ be given by Theorem~\ref{thm:non-extremal} on input $\beta = \min\{ \alpha_0/4, \beta_{\ref{thm:extremal}} \}$.
	Moreover, let $C_{\ref{thm:sublinear}}$ be given by Theorem~\ref{thm:sublinear}.
	Define $C = \max \{ C_{\ref{thm:extremal}},C_{\ref{thm:sublinear}} \}$ and $\gamma=\min \{ \gamma_{\ref{thm:extremal}},\gamma_{\ref{thm:non-extremal}} \}$.
	
	Let $G$ be any $n$-vertex graph and $p \ge C \log n/n$, and define $m= \min \{ \delta(G), \lfloor n/3 \rfloor \}$.
	If $m \le n/256$, then we get from Theorem~\ref{thm:sublinear} that a.a.s.~$G \cup G(n,p)$ contains at least $m$ pairwise vertex-disjoint triangles, as $C \ge C_{\ref{thm:sublinear}}$.
	Otherwise, $m > n/256$ and we can choose $\alpha \in (\alpha_0,1/3]$ such that $(\alpha-\gamma)n \le m \le \alpha n$.
	If $G$ is $(\alpha,\beta)$-stable, then $G$ is also $(\alpha,\beta_{\ref{thm:extremal}})$-stable and, by Theorem~\ref{thm:extremal}, there are a.a.s.~at least $\min \{ \delta(G),\lfloor \alpha n\rfloor  \} \ge m$ pairwise vertex-disjoint triangles in $G \cup G(n,p)$, as $\alpha_0 <\alpha \le 1/3$ and $C \ge C_{\ref{thm:extremal}}$.
	Otherwise, $G$ is not $(\alpha,\beta)$-stable and, by Theorem~\ref{thm:non-extremal}, a.a.s.~$G \cup G(n,p)$ contains at least $ \min \{ \alpha n, \lfloor n/3 \rfloor \} \ge m$ pairwise vertex-disjoint triangles, as $p = \omega(1/n)$.
\end{proof}

\section{Tools}
\label{sec:tools}

We will repeatedly use the following concentration inequality due to Chernoff (see e.g.~\cite[Corollaries 2.3 and 2.4]{JLR}).
\begin{lemma}[Chernoff's inequality]
	\label{lem:chernoff}
	Let $X$ be the sum of independent binomial random variables, then for any $\delta \in (0,1)$ we have
	\[ \PP \left[ |X-\EE[X]| \ge \delta \, \EE[X] \right] \le 2 \exp \left( -\frac{\delta^2}{3} \EE[X] \right). \]
	Moreover, for any $k \ge 7\, \EE[X]$ we have $\PP[X > k] \le \exp (-k)$. 
\end{lemma}

The following lemma allows us to find specific triangles in a dense graph with additional random edges.
The proof is a standard application of Janson's inequality (see e.g.~\cite[Theorem 2.18]{JLR}) and it is included in Appendix~\ref{sec:appendix}.
\begin{lemma}
	\label{lem:triangle}
	For any $d>0$ there exists $C>0$ such that the following holds.
	Let $U,V,W$ be three sets of vertices of size $n$, $G$ be a bipartite graph on $(U,W)$ with $e(G) \ge d n^2$, $p > C/n$ and $G(U \cup W,V,p)$ be the random bipartite graph.
	Then with probability at least $1-2^{-4n/d}$ there is a triangle in $G \cup G(U \cup W,V,p)$ with one vertex in each of $U,V,W$.
\end{lemma}

We will use Szemer\'edi's Regularity Lemma~\cite{szem76} and some of its consequences.
Before stating it, we introduce the relevant terminology.
The \emph{density} of a pair $(A,B)$ of disjoint sets of vertices is defined by 
\begin{align*}
	d(A,B)= \frac{e(A,B)}{|A|\cdot |B|}
\end{align*}
and the pair $(A,B)$ is called \emph{$\varepsilon$-regular}, if for all sets $X \subseteq A$ and $Y \subseteq B$ with $|X| \ge \varepsilon |A|$ and $|Y| \ge \varepsilon |B|$ we have $|d(A,B)-d(X,Y)| \le \varepsilon$.
Without further mentioning it, we will repeatedly use that when $\eps<1/2$ any $A' \subseteq A$ with $|A'| \ge |A|/2$ and $B' \subseteq B$ with $|B'| \ge |B|/2$ give a $2\eps$-regular pair $(A',B')$ with density at least $d(A,B)-\eps$.

We will also use the following well-known result that follows from the definition.

\begin{lemma}[Minimum Degree Lemma]
	\label{lem:MDL}
	Let $(A,B)$ be an $\eps$-regular pair with $d(A,B)=d$.
	Then for every $Y \subseteq B$ with $|Y| \ge \eps |B|$, the number of vertices from $A$ with degree into $Y$ less than $(d-\eps)|Y|$ is at most $\eps |A|$.
\end{lemma}

With $d \in [0,1)$ a pair $(A,B)$ is called \emph{$(\varepsilon,d)$-super-regular}, if for all sets $X \subseteq A$ and $Y \subseteq B$ with $|X| \ge \varepsilon |A|$ and $|Y| \ge \varepsilon |B|$ we have $d(X,Y) \ge d$ and $\deg(a) \ge d|B|$ for all $a \in A$ and $\deg(b) \ge d|A|$ for all $b \in B$.
It is easy to prove with Hall's Theorem that a super-regular pair with parts of the same size contains a perfect matching.
\begin{lemma}
	\label{lem:blow-up}
	For any $d>0$ there exists $\eps>0$ such that any $(\eps,d)$-super-regular pair $(U,V)$ with $|U|=|V|$ contains a perfect matching.
\end{lemma}

We will use the following well-known degree form of the regularity lemma that can be derived from the original version~\cite{szem76}.

\begin{lemma}[\cite{simon96}]
	\label{lem:reg}
	For every $\varepsilon>0$ and integer $t_0$ there exists an integer $T>t_0$ such that for any graph $G$ on at least $T$ vertices and $d \in [0,1]$ there is a partition of $V(G)$ into $t_0 < t+1 \le T$ sets $V_0,\dots,V_{t}$ and a subgraph $G'$ of $G$ such that
	\begin{enumerate}[label=\upshape(P\arabic*)]
		\item \label{prop:size} $|V_i| = |V_j|$ for all $1 \le i,j \le t$ and $|V_0|\le \varepsilon |V(G)|$,
		\item \label{prop:degree} $\deg_{G'}(v) \ge \deg_G(v) - (d+\varepsilon) |V(G)|$ for all $v \in V(G)$,
		\item \label{prop:indepen} the set $V_i$ is independent in $G'$ for all $1 \le i \le t$,
		\item \label{prop:regular} for all $1 \le i < j \le t$, the pair $(V_i,V_j)$ is $\varepsilon$-regular in $G'$ and has density either $0$ or at least $d$.
	\end{enumerate}
\end{lemma}

The sets $V_1, \dots, V_t$ are also called clusters and we refer to $V_0$ as the set of exceptional vertices.
We call a partition $V_0, \dots, V_t$, which satisfies~\ref{prop:size}--\ref{prop:regular}, a \emph{$(\eps,d)$-regular partition of $G$}.
Given this partition, we define the \emph{$(\eps,d)$-reduced graph} $R$ for $G$, that is the graph on the vertex set $[t]$, where $ij$ is an edge if and only if $(V_i,V_j)$ is an $\eps$-regular pair in $G'$ and has density at least $d$.

We will also use the following result on perfect matchings in random subgraphs of bipartite graphs with large minimum degree.
For any given graph $G$ we denote by $G_p$ the random graph model, where we keep each edge of $G$ with probability $p$, independently from all other choices.

\begin{lemma}
	\label{lem:matching}
	For any $\eps>0$ there exists $C>0$ such that the following holds for any bipartite graph $G$ with partition classes $|U|=|W|=n$ and minimum degree $\delta(G)\ge (1/2+\eps) n$.
	With $p\ge C \log n/n$ there is a.a.s. a perfect matching in $G_p$.
\end{lemma}

The proof is standard and closely follows the proof for the perfect matching threshold in the random bipartite graph $G(n,n,p)$ (see~\cite[Theorem 4.1]{JLR}).
For completeness we include it in Appendix~\ref{sec:appendix}.

\section{Proof Overview and main Lemmas}
\label{sec:overview}

In this section we sketch the ideas behind our proof of Theorem~\ref{thm:non-extremal},~\ref{thm:extremal}, and~\ref{thm:sublinear}, and we give the statements of the lemmas we use.
For simplicity, when outlining the proof of Theorem~\ref{thm:non-extremal} and~\ref{thm:extremal} we assume $\alpha=1/3$, $n$ is a multiple of $3$, and $G$ is an $n$-vertex graph with minimum degree $\delta(G) \ge n/3$, in which case both theorems give a triangle factor in $G \cup G(n,p)$.

\subsection{Extremal case.} Assume that $G$ is $(1/3,\beta)$-stable and let $p\ge C \log n /n$.
The definition of stability (Definition~\ref{def:stability}) gives a partition of $V(G)$ into $A \cup B$ where the size of $B$ is roughly the double of the size of $A$, there is a minimum degree condition between $A$ and $B$, and in each part all but at most a few vertices see all but at most few a vertices of the other part.
Our proof will follow three steps.
Firstly, we find a collection of triangles $\cT_1$, such that after removing the triangles of $\cT_1$, we are left with two sets $A_1=A \setminus V(\cT_1)$ and $B_1=B \setminus V(\cT_1)$ with $|B_1| = 2|A_1|$.
The way we find these triangles depends on the sizes of $A$ and $B$ and we will use two different approaches when $|B| > 2n/3$ and  $|B| \le 2n/3$.
In particular when $|B| > 2n/3$ we need to find some triangles entirely within $B$, just using the minimum degree $n/3-|A|$ and random edges.
For that we will use Theorem~\ref{thm:sublinear}.

Our second step is to cover the vertices in $A_1$ and $B_1$ that do not have a high degree to the other part; this will give two collections of triangles $\cT_2$ and $\cT_3$.
Each such triangle has one vertex in $A_1$ and two vertices in $B_1$ so that we still have $|B_2| = 2 |A_2|$, where $A_2=A_1 \setminus V(\cT_2 \cup \cT_3)$ and $B_2=B_1 \setminus V(\cT_2 \cup \cT_3)$. 
Moreover at this point each vertex sees all but at most a few vertices of the other part.
We are now ready for the last step.
We split $B_2$ arbitrarily into two subsets $B_2'$ and $B_2''$ of equal size and we obtain that $(B_2'$, $A_2$, $B_2'')$ is a super-regular cherry, i.e. both $(B_2',A_2)$ and $(B_2'',A_2)$ are super-regular pairs.
We want to find a triangle factor covering the cherry, with the help of random edges between $B_2'$ and $B_2''$.
The next lemma, which encapsulates the main idea of our paper, takes care of this and will be proved in Section~\ref{sec:auxiliary}.

\begin{lemma}
	\label{lem:tripartite}
	For any $0<d<1$ there exist $\eps>0$ and $C>0$ such that the following holds.
	Let $U,V,W$ be sets of size $n$, let $(V,U)$ and $(V,W)$ be $(\varepsilon,d)$-super-regular pairs and let $G(U,W,p)$ be a random bipartite graph with $p\ge C \log n/n$. Then a.a.s.~there exists a triangle factor.
\end{lemma}

Thus we are able to cover the cherry with a triangle factor $\cT_4$ and we conclude observing that the collection of triangles $\cT_1 \cup \cT_2 \cup \cT_3 \cup \cT_4$ gives a triangle factor in $G \cup G(n,p)$.

\subsection{Non-extremal case.} Before giving an overview of Theorem~\ref{thm:non-extremal}, it is worth to make some comments about Lemma~\ref{lem:tripartite}. 
We point out that the probability $p$ cannot be significantly lowered.
Indeed, a triangle factor in the setting of Lemma~\ref{lem:tripartite} gives a perfect matching in the random bipartite graph on vertex set $U \cup W$ and this is a.a.s.~not possible with $p\le \frac{1}{2} \log n/n$~\cite[Theorem~4.1]{JLR}.
However, we would like to be able to find a triangle factor in a super-regular cherry with the help of the random edges also in the proof of Theorem~\ref{thm:non-extremal}, where we claimed that when the graph $G$ is not $(1/3,\beta)$-stable, already $p \ge C/n$ is sufficient.
For that we will use the following variation of Lemma~\ref{lem:tripartite}, where the improvement on the probability comes from the assumption that the super-regular cherry $U,V,W$ is a bit unbalanced as the sizes of $U$ and $W$ are smaller than the size of $V$, and thus the random bipartite graph on vertex set $U \cup W$ will be used to build a large matching covering all but a small linear fraction of vertices, which is possible already with $p \ge C/n$.
Note that here we need to use random edges within $V$.

\begin{lemma}
	\label{lem:tripartite2}
	For any $0<\delta'\le d<1$ there exist $\delta_0,\delta,\eps$ with $\delta'\ge \delta_0>\delta>\eps>0$ and $C>0$ such that the following holds.
	Let $U,V,W$ be sets of size $|V|=n$ and $(1-\delta_0) n \le |U|=|W| \le (1-\delta)n$ where $|V|+|U|+|W| \equiv 0 \pmod 3$.
	Further, let $(V,U)$ and $(V,W)$ be $(\varepsilon,d)$-super-regular pairs and let $G(V,p)$, $G(U,W,p)$ be random graphs with $p\ge C /n$.
	Then a.a.s.~there exists a triangle factor.
\end{lemma}

From Lemma~\ref{lem:tripartite2} we also derive the following result about the existence of a triangle factor in a super-regular pair edge, again with the help of the random edges.
Lemmas~\ref{lem:tripartite2}, and~\ref{lem:bipartite} will be proved in Section~\ref{sec:auxiliary}.

\begin{lemma}
	\label{lem:bipartite}
	For any $0<d<1$ there exist $\eps>0$ and $C>0$ such the following holds for sets $U,V$ of size $|V|=n$ and $3n/4 \le |U| \le n$ where $|V|+|U| \equiv 0 \pmod 3$.
	If $(U,V)$ is an $(\varepsilon,d)$-super-regular pair and $G(U,p)$ and $G(V,p)$ are random graphs with $p\ge C /n$, then a.a.s.~there exists a triangle factor.
\end{lemma}

Now we turn to the overview of the proof for Theorem~\ref{thm:non-extremal}. 
Assume that $G$ is not $(1/3,\beta)$-stable and let $p\ge C /n$.
We apply the regularity lemma to $G$ and obtain the reduced graph $R$.
By adjusting an argument of the fourth author with Balogh and Mousset~\cite{BMS_cover}, we can prove the following stability result.

\begin{lemma}
	\label{lem:stable_cluster}
	For any $0<\beta<1/12$ there exists $d>0$ such that the following holds for any $0<\eps<d/4$, $4 \beta \le \alpha \le 1/3$, and $t \ge 10/d$ .
	Let $G$ be an $n$-vertex graph with minimum degree $\delta(G) \ge (\alpha -d/2)n$ that is not $(\alpha,\beta)$-stable and let $R$ be the $(\eps,d)$-reduced graph for some $(\eps,d)$-regular partition $V_0,\dots,V_t$ of $G$.
	Then $R$ contains a matching $M$ of size $(\alpha+2d)t$.
\end{lemma}

For completeness, we give the proof in Appendix~\ref{sec:appendix}.
It follows that we can cover the vertices of $R$ with cherries $K_{1,2}$ and matching edges $K_{1,1}$, such that there are not too many cherries.\footnote{We remark that covers of the reduced graph by stars were used in~\cite{balogh2019tilings,krivelevich2017bounded} and this inspired our approach. For our purposes it is necessary that we cover the reduced graph by cherries and matching edges. Furthermore, in contrast to~\cite{balogh2019tilings} we can not rely on any triangles in the random graph but need to use the edges of $G$ to build them.}
Before we can apply Lemma~\ref{lem:tripartite2} to each cherry and Lemma~\ref{lem:bipartite} to each matching edge, some preliminary steps are needed.
We remove some vertices from each cherry to make it unbalanced and ensure that both edges are super-regular.
Then we cover all vertices that are not contained in any of the cherries or edges by finding a collection of triangles $\cT_1$.
We construct another collection of triangles $\cT_2$ to ensure that in each cherry the relations between the three sets are as required by Lemma~\ref{lem:tripartite2}.
For constructing $\cT_1$ and $\cT_2$ we will mainly rely on the minimum degree condition of $G$ and the fact that in the probability $p \ge C/n$, the constant $C$ can be chosen large enough so that a.a.s.~the following holds: each linear-sized set contains a random edge and for any not too small part of a regular pair and a linear-sized set there is a triangle containing an edge form the pair and the third vertex from the set.
Finally, we can use Lemma~\ref{lem:tripartite2} and Lemma~\ref{lem:bipartite} to cover the remaining vertices with a collection of triangles $\cT_3$.
Together $\cT_1 \cup \cT_2 \cup \cT_3$ gives a triangle factor in $G \cup G(n,p)$.

We mention already now that when $\alpha$ is sufficiently smaller than $1/3$ and the condition on the minimum degree of $G$ reads as $\delta(G) \ge (\alpha-\gamma)n$, many of the steps outlined above are not necessary.
In this case indeed we do not have to cover all graph with triangles and we only want to find $\alpha n$ pairwise vertex-disjoint triangles in $G \cup G(n,p)$.
We will see that an application of Lemma~\ref{lem:tripartite2} and Lemma~\ref{lem:bipartite} to the cherries and the matching edges found at the beginning (after having made them super-regular and suitable for Lemma~\ref{lem:tripartite2}) is already enough to find these $\alpha n$ triangles.

\subsection{Sublinear case.}
Assume $G$ is an $n$-vertex graph with minimum degree $\delta(G) \ge m$ and let $p\ge C \log n/n$.
We want to show that a.a.s.~there exist $m$ pairwise vertex-disjoint triangles in $G \cup G(n,p)$.
Any vertex of large degree in $G$ can easily be covered by a triangle later, so we assume an upper bound on the maximum degree $\Delta(G)$ of the graph $G$.
With this condition, we split the proof in three ranges for the value of $m$: $1 \le m \le (\log n)^3$, $(\log n)^3 \le m \le \sqrt{n}$, and $\sqrt{n} \le m \le n/256$.
If $1 \le m \le (\log n)^3$ a.a.s.~$m$ pairwise vertex-disjoint triangles already exist in $G(n,p)$.
If $(\log n)^3 \le m \le \sqrt{n}$ we will find many large enough vertex-disjoint stars in $G$ (see Lemma~\ref{lem:manystars}) and a.a.s.~at least $m$ of them will be completed to triangles using edges of $G(n,p)$ (see Proposition~\ref{prop:lesqrt}).
However if $m > \sqrt{n}$ we cannot hope to find $m$ large enough vertex-disjoint stars and instead we will apply a greedy strategy using that a.a.s.~every vertex has an edge in its neighbourhood (see Proposition~\ref{prop:gesqrt}).

\section{Proof of the Extremal Theorem}
\label{sec:extremal}

\begin{proof}[Proof of Theorem~\ref{thm:extremal}]
	Let $0 < \alpha_0 \le 1/3$ and choose $d=1/2$.
	Let $\varepsilon>0$ and $C_{\ref{lem:tripartite}}>0$ be given by Lemma~\ref{lem:tripartite} on input $d$.
	We can assume $\eps < 1/5$ and then choose $0<\beta< \alpha_0\, \eps /36$ and $0<\gamma<\beta/11$. 
	With $C_{\ref{thm:sublinear}}$ given by Theorem~\ref{thm:sublinear}, let $C \ge 2 C_{\ref{thm:sublinear}} + 1 + 2 C_{\ref{lem:tripartite}}/\alpha_0$.
	Finally, let $\alpha_0 \le \alpha  \le 1/3$.
	
	Given $n$, let $p\ge C \log n/n$.
	With our choice of $C$, we can reveal $G(n,p)$ in three rounds $G_1 \sim G(n,2 C_{\ref{thm:sublinear}} \log n/n)$, $G_2 \sim G(n,\log n/n)$, and $G_3 \sim G(n,\tfrac{2 C_{\ref{lem:tripartite}}}{\alpha_0} \log n/n)$.
	We will only know later in which subset we will use $G_1$ and $G_3$, but we have that a.a.s.~there is an edge of $G_2$ between any two not necessarily disjoint sets of size $\beta n$.
	Indeed, fixed two such sets, the probability that there is no edge of $G_2$ is at most $(1-\log n/n)^{(\beta n)^2} \le \exp(-\beta^2 n \log n)$, and we conclude by an union bound over the at most $2^{2n}$ choices for the two sets.
	Now let $G$ be an $n$-vertex graph with minimum degree $\delta(G) \ge \left( \alpha-\gamma \right) n$ that is $(\alpha,\beta)$-stable and define $m_0= \max\{ n/3 -\delta(G), n/3-\lfloor \alpha n \rfloor \}$.
	Our goal is to a.a.s.~find pairwise vertex-disjoint triangles in $G \cup G(n,p)$ such that at most $3 m_0$ vertices are left uncovered.
	
	To aid with calculations we let $\kappa = n/3 - \lceil \alpha n \rceil$ and observe that $\kappa \in \{ 0,-1/3,-2/3 \}$ if $\alpha=1/3$ and that $\kappa>0$ if $\alpha<1/3$ and $n$ large enough.
	Also note that $m_0-\kappa$ is an integer and that $3 \kappa = \lfloor (1-\alpha) n \rfloor - 2 \lceil \alpha n \rceil$.
	With this we set $w = \max\{ 3 \kappa, 0\}$.
	As $G$ is $(\alpha, \beta)$-stable we get a partition of $V(G)$ into  sets $A$ and $B$ satisfying the conditions of Definition~\ref{def:stability}.
	
	\begin{claim}
		\label{claim:B=2A}
		There a.a.s.~are a collection of triangles $\cT_1$ in $G \cup G_1 \cup G_2$ with $|\cT_1| \le \beta n$ and a set $W \subseteq V(G) \setminus V(\cT_1)$ with $|W| \le 3 m_0 - w$ such that the following holds.
		For $A_1=A \setminus (V(\cT_1) \cup W)$ and $B_1=B \setminus (V(\cT_1) \cup W)$, we have that $|A_1| \le \lceil \alpha n \rceil$, $|B_1|=2|A_1|+w$, the minimum degree between $A_1$ and $B_1$ is at least $\alpha n/5$, all but at most $\beta n$ vertices of $A_1$ have degree at least $|B_1| - \beta n$ into $B_1$, and all but at most $\beta n$ vertices of $B_1$ have degree at least $|A_1| - \beta n$ into $A_1$.
	\end{claim}
	
	The sets $A_1$ and $B_1$ partition $V(G) \setminus (V(\cT_1) \cup W)$ and, after proving Claim~\ref{claim:B=2A}, we will cover all but $w$ vertices from $A_1 \cup B_1$ with additional triangles.
	Hence, if we manage to find these triangles, we have covered all but $|W|+w \le 3 m_0$ vertices, as desired.
	We remark for later that $|W| \le 3m_0-w \le 4 \gamma n$.
	
	\begin{claimproof}[Proof of Claim~\ref{claim:B=2A}]
		We have either $|B| > \lfloor (1-\alpha) n \rfloor $ or $|A| \ge \lceil \alpha n \rceil$.
		First suppose that we are in the first case, where $|B|=\lfloor (1-\alpha)n \rfloor +m$ for some $1 \le m \le \beta n$ (and $|A|=\lceil \alpha n \rceil - m$), and note that
		\[|B|-2|A|= n-3 \lceil \alpha n\rceil +3m =3m+3\kappa >0 \,. \]
		If $1 \le m \le m_0 - \kappa$, then $0 < 3m \le 3 m_0 - 3 \kappa $ and we let $W$ be any set with $\min \{3m,3m+3 \kappa\} \le 3m_0-w$ vertices from $B$.
		Then with the choice of $\cT_1=\emptyset$, we have that the sets $A_1=A$ and $B_1=B \setminus W$ partition $V(G) \setminus W$, and $|A_1|=\lceil \alpha n \rceil-m$ and $|B_1|=|B|-\min\{ 3m,3m+3 \kappa \}=2|A_1|+ w$.
		If on the other hand $m_0 < m + \kappa$, then 
		\[\delta(G[B]) \ge \delta(G) - |A| \ge (n/3-m_0)-(\lceil \alpha n \rceil-m) = m-m_0+\kappa>0\]
		and we observe that $m-m_0+\kappa$ is an integer.
		Moreover $m-m_0+\kappa \le m \le |B|/256$, where we use $m_0-\kappa \ge 0$,  $m \le \beta n \le \alpha_0 \eps n/36 \le n/(3 \cdot 5 \cdot 36)$ and $|B| = \lfloor (1-\alpha)n \rfloor + m \ge n/2 +m$.
		Thus, by Theorem~\ref{thm:sublinear} and as $2 C_{\ref{thm:sublinear}} \log n/n \ge C_{\ref{thm:sublinear}} \log |B|/ |B|$ we a.a.s.~find $m-m_0+\kappa$ pairwise vertex-disjoint triangles in $(G \cup G_1)[B]$.
		Denote by $\cT_1$ the collection of these $m-m_0+\kappa$ triangles.
		Let $W$ be any set of $3m_0-w$ vertices from $B$ not covered by any triangle in $\cT_1$.
		Then the sets $A_1=A$ and $B_1=B \setminus (V(\cT_1) \cup W)$ partition $V(G) \setminus (V(\cT_1) \cup W)$, and we have $|A_1|=\lceil \alpha n \rceil-m$ and 
		\[|B_1|=|B|-3(m-m_0+\kappa)-(3 m_0 - w)=2|A_1| + w\, .\]
		
		It remains to consider the second case, where $|A|= \lceil \alpha n \rceil+m$ for some $0 \le m \le \beta n$.
		First, we greedily pick $m$ pairwise vertex-disjoint triangles in $G \cup G_2$ each with two vertices in $A$ and one vertex in $B$.
		Indeed during the process, there is always a vertex $v$ in $B$, not yet contained in a triangle, with at least $\deg(v,A)-2m \ge (\alpha/4-2\beta)n \ge \beta n$ uncovered neighbours in $A$ in the graph $G$.
		By the property assumed in $G_2$ we can then find an edge within these neighbours of $v$ to get a triangle.
		Denote by $\cT_1$ the collection of these $m$ triangles.
		
		If $\kappa \ge 0$, then, with the choice of $W=\emptyset$, we have that $A_1=A \setminus V(\cT_1)$ and $B_1=B \setminus V(\cT_1)$ partition $V(G) \setminus V(\cT_1) $ and \[|B_1|=|B|-|\cT_1|= \lfloor (1-\alpha)n \rfloor - 2m = 2(\lceil \alpha n\rceil -m)+ 3 \kappa=2|A_1|+w.\]

		If $\kappa < 0$, we additionally pick a set $W$ of vertices not covered by triangles from $\cT_1$, such that $|W|=1$, $|W \cap A|=1$, $|W \cap B|=0$ if $\kappa=-2/3$, and $|W|=2$, $|W \cap A|=|W\cap B|=1$ if $\kappa=-1/3$.
		Then, the sets $A_1=A \setminus (V(\cT_1) \cup W)$ and $B_1=B \setminus (V(\cT_1) \cup W)$ partition $V(G) \setminus (V(\cT_1) \cup W)$, and $|B_1|=2|A_1|+w$.
		Indeed, if $\kappa=-2/3$ we have $|A_1|=|A|-2|\cT_1|-1=\lceil \alpha n\rceil - m-1$ and \[|B_1|=|B|-|\cT_1|= \lfloor (1-\alpha)n \rfloor - 2m = 2 (\lceil \alpha n\rceil -m) + 3 \kappa = 2|A_1|\]
		and if $\kappa=-1/3$ we have $|A_1|=|A|-2|\cT_1|-1=\lceil \alpha n\rceil - m-1$ and \[|B_1|=|B|-|\cT_1|-1= \lfloor (1-\alpha)n \rfloor - 2m -1 = 2 (\lceil \alpha n\rceil -m)  + 3 \kappa -1= 2|A_1|\, .\]
		
		Observe, that in both the first and the second case $|B_1|=2|A_1|+w$ and $|W| \le 3m_0-w$.
		Moreover, as we remove at most $3m_0-w \le 4 \gamma n \le \alpha n/20$ vertices from each $A$ and $B$, the minimum degree between $A_1$ and $B_1$ is at least $\alpha n/4 - \alpha n/20 = \alpha n/5$.
		The other conditions on the degrees between $A_1$ and $B_1$ are clearly satisfied, because for all but at most $\beta n$ vertices from each set there are still at most $\beta n$ non-neighbours in the other set.
		The bounds $|\cT| \le \beta n$ and $|A_1| \le \lceil \alpha n \rceil$ also hold in all cases.
	\end{claimproof}

	We want to cover all but $w$ vertices in $A_1 \cup B_1$ and we start from those vertices in $A_1$ and $B_1$ that do not have a high degree to the other part.
	We will always cover them with triangles with one vertex in $A_1$ and two vertices in $B_1$ to ensure that the relation between the number of vertices remaining in $A_1$ and $B_1$ does not change.
	Let
	\[
	\tilde{A_1} = \{v \in A_1: \deg(v,B_1) \le |B_1|-9\beta n \} \text{ and }
	\tilde{B_1} = \{v \in B_1: \deg(v,A_1) \le |A_1|-9\beta n \},
	\]
	and observe that $|\tilde{A_1}|,|\tilde{B_1}| \le \beta n$.
	
	We claim that a.a.s.~we can greedily pick pairwise vertex-disjoint triangles in $(G \cup G_2)[A_1 \cup B_1]$ that cover all vertices of $\tilde{A_1}$, with each triangle having one vertex in $\tilde{A_1}$ and two vertices in $B_1 \setminus \tilde{B_1}$.
	Indeed, at each step during the process, an uncovered vertex $v$ in $\tilde{A_1}$ has at least $\deg(v,B_1)- |\tilde{B_1}| - 2 |\tilde{A_1}| \ge (\alpha/5-3 \beta)n \ge \beta n$ uncovered neighbours in $B_1 \setminus \tilde{B_1}$ in the graph $G$.
	We then find an edge of $G_2$ within these neighbours of $v$ and build a triangle.
	Denote by $\cT_2$ the collection of these triangles and note that $|\cT_2| \le \beta n$.
	
	Observe that at this point $2 |\cT_2| \le 2\beta n $ vertices of $B_1 \setminus \tilde{B_1}$ have already been covered.
	We claim that a.a.s.~we can greedly pick pairwise vertex-disjoint triangles in $(G \cup G_2)[(A_1 \cup B_1) \setminus V(\cT_2)]$ that cover all vertices of $\tilde{B_1}$, where each triangle has one vertex in $A_1 \setminus \tilde{A_1}$, one vertex in $\tilde{B_1}$ and one vertex in $B_1 \setminus \tilde{B_1}$.
	Indeed, at each step during the process, an uncovered vertex $v$ in $\tilde{B_1}$ has at least $\deg(v,A_1) - |\tilde{A_1}| - |\tilde{B_1}| \ge (\alpha/5-2 \beta)n \ge \beta n$ uncovered neighbours in $A_1 \setminus \tilde{A_1}$ in the graph $G$ and at least 
	\begin{align*}
		\delta(G)& - 3|\cT_1| - |W| - \deg(v,A_1) - 2 |\cT_2| - 2|\tilde{B_1}| \\
		&\ge (\alpha - \gamma)n - 3\beta n - 4 \gamma n - (\lceil \alpha n \rceil -9 \beta n) - 4 \beta n \ge \beta n
	\end{align*}
	uncovered neighbours in $B_1 \setminus \tilde{B_1}$ in the graph $G$.
	We then find an edge of $G_2$ between these two neighbourhood sets to get a triangle.
	Denote by $\cT_3$ the collection of these triangles and note that $|\cT_3| \le \beta n$.
	
	The sets $A_2=A_1 \setminus V(\cT_2 \cup \cT_3)$ and $B_2=B_1 \setminus V(\cT_2 \cup \cT_3)$ give a partition of the remaining vertices in $V(G) \setminus (V(\cT_1) \cup V(\cT_2) \cup V(\cT_3) \cup W)$.
	We have 
	\[|A_2|\ge |A| - 2|\cT_1| - |\cT_2| - |\cT_3| - |W|  \ge \alpha n - 5 \beta n - 4 \gamma n \ge \alpha n/2\, \]
	and $|B_2|=2|A_2| + w$. 
	Moreover, the degree from $A_2$ to $B_2$ is at least $|B_1|-9 \beta n - 2|\cT_2 \cup \cT_3| = |B_2| - 9 \beta n$ and the degree from $B_2$ to $A_2$ is at least $|A_1|- 9 \beta n - |\cT_2 \cup \cT_3| = |A_2| - 9 \beta n$. 
	We partition $B_2$ arbitrarily into three subsets $B_2'$, $B_2''$, and $W'$ of size $|B_2'|=|B_2''|=|A_2|$ and $|W'|=|B_2|-2|A_2|=w$.
	Then the degree from $A_2$ to $B_2'$ and the degree from $A_2$ to $B_2''$ are at least $|B_2'| - 9 \beta n$.
	
	We claim that the pair $(A_2,B_2')$ is $(\varepsilon,1/2)$-super-regular, with $\varepsilon$ being chosen as stated at the beginning of the proof.
	Indeed for all $X \subset A_2$ and $Y \subset B_2'$ with $|X| \ge \varepsilon |A_2|$ and $|Y| \ge \varepsilon |B_2'|$, we have
	\[e(X,Y) \ge |X| (|Y| - 9 \beta n) \ge \frac{1}{2} |X| |Y| \]
	and $\deg(a,B_2') \ge |B_2'| - 9 \beta n \ge |B_2'|/2$ for all $a \in A_2$, and $\deg(b,A_2) \ge |A_2| - 9 \beta n \ge |A_2|/2$ for all $b \in B_2'$, where for all inequalities we use $\beta \le \alpha/36$.
	For the same reason, the pair $(A_2,B_2'')$ is $(\varepsilon,1/2)$-super-regular as well.
	
	Now, as $\tfrac{2C_{\ref{lem:tripartite}}}{\alpha_0} \log n/n \ge C_{\ref{lem:tripartite}} \log |A_2|/|A_2|$, we can apply Lemma~\ref{lem:tripartite} to $A_2$, $B_2'$ and $B_2''$, with $d=1/2$ and $\varepsilon$, and a.a.s.~get a triangle factor $\cT_4$ in $(G \cup G_3) [V']$, where $V'=V(G) \setminus (V(\cT_1) \cup V(\cT_2) \cup V(\cT_3) \cup W \cup W')$.
	Then $\cT_1 \cup \cT_2 \cup \cT_3 \cup \cT_4$ contains at least
	\[(n - |W| - |W'|)/3 \ge n/3- m_0 \ge \min \{ \delta(G), \lfloor \alpha n \rfloor \}\]
	pairwise vertex-disjoint triangles covering $V(G) \setminus (W \cup W')$.
\end{proof}

We point out that under certain conditions our proof of Theorem~\ref{thm:extremal} gives more triangles.
When $\alpha < 1/3$ and $|A| \ge \alpha n$, as $|W| \le  3m_0-w$, we get $\lceil \alpha n \rceil$ pairwise vertex-disjoint triangles in $G \cup G(n,p)$, even when $\delta(G) < \alpha n$.
Similarly, when $\alpha = 1/3$ and $|A| \ge n/3$, as $|W| \le 2$, we get $\lfloor n/3 \rfloor$ pairwise vertex-disjoint triangles in $G \cup G(n,p)$, even when $\delta(G) < n/3$.
Moreover, for any value of $\alpha$, when $|A|-n/3>0$ is linear in $n$ we could use Lemma~\ref{lem:tripartite2} instead of Lemma~\ref{lem:tripartite} to avoid the $\log n$-factor in the probability.

\section{Proof of the Stability Theorem}
\label{sec:non-extremal}

\begin{proof}[Proof of Theorem~\ref{thm:non-extremal}]
	We start by defining necessary constants. Given $0<\beta<1/12$, let $d>0$ be obtained from Lemma~\ref{lem:stable_cluster}, and set $\gamma=d/2$ and $t_0=11/d$.
	Next, we take any $0<\delta' <  160^{-2} d^2$ and use Lemma~\ref{lem:tripartite2} on input $d/2$ and $\delta'$ to obtain $\delta_0, \delta, \eps'$ with $\delta' \ge \delta_0>\delta>\eps'>0$ and $C_1$.
	Additionally we assume that $C_1$ is large enough and $\eps'$ is small enough for Lemma~\ref{lem:bipartite} to hold with input $d/2$.
	Finally, let $C_2$ be given by Lemma~\ref{lem:triangle} on input $d/2$.
	We let $0<\eps \le \eps'/2$.
	In summary, the dependencies between our constants are as follows:
	\[ 
	\eps \ll \eps' < \delta < \delta_0 \le \delta' \ll d \ll \beta < \frac{1}{12}
	\quad \text{and} \quad 
	\frac{1}{t_0}, \gamma \ll d. 
	\]
	We apply Lemma~\ref{lem:reg} with $\eps$ and $t_0$ to obtain $T$. We take $C$ large enough such that, for $p \ge C/n$, the random graph $G(n,p)$ contains the union $G_1 \cup G_2 \cup G_3$, where $G_1 \sim G(n,2C_1T/n)$, $G_2 \sim G(n,4C_2T/(dn))$, and $G_3 \sim G(n,96 T^2/(d^2n))$.
	
	Now, for any $\alpha$ with $4 \beta \le \alpha \le 1/3$, let $G$ be an $n$-vertex graph on the vertex set $V$ with minimum degree $\delta(G) \ge (\alpha-\gamma)n$ that is not $(\alpha,\beta)$-stable.
	With the regularity lemma (Lemma~\ref{lem:reg}) applied to $G$, we get $G'$, $t_0 < t+1 \le T$ and a partition $V_0,\dots,V_t$ of $V(G)$ such that~\ref{prop:size}--\ref{prop:regular} hold.
	Define $n_0=|V_1|=|V_2|=\dots =|V_t|$ and observe that $(1-\eps)n/t \le n_0 \le n/t$.
	We denote by $R$ the $(\eps,d)$-reduced graph for $G$, that is, the graph on the vertex set $[t]$ with edges $ij$ corresponding to $\eps$-regular pairs $(V_i,V_j)$ of density at least~$d$ in $G'$.
	We observe that the minimum degree of $R$ satisfies $\delta(R) \ge (\alpha-2d) t$ because, otherwise, there would be vertices with degree at most $(\alpha-2d)t(n/t)+\eps n < (\alpha - \gamma)n - (d+\eps)n$ in $G'$, contradicting~\ref{prop:degree}.
	
	The purpose of $G_1$ will become clear later, but we describe some useful properties of $G_2$ and $G_3$ now.
	Let $U$ and $W$ be any two clusters that give an edge in $R$, $V$ any cluster, and $U' \subseteq U$, $W'\subseteq W$, $V'\subseteq V$ three pairwise disjoint subsets each of size $dn_0/2$. Then, with $G_2$ and as $4C_2T/(dn) \ge 2C_2/(dn_0)$, by Lemma~\ref{lem:triangle} we have that with probability at least $1-2^{-4(dn_0/2)/(d/2)}=1-2^{-4n_0}$
	\begin{align}
		\label{random_graph_G2} 
		\text{there is a triangle in $G \cup G_2$ with one vertex in each set $U'$, $W'$, $V'$.}
	\end{align}
	With a union bound over the at most $t^32^{3n_0}$ choices for $U$, $W$, $V$ and $U'$, $W'$, $V'$, we conclude that a.a.s.~\eqref{random_graph_G2} holds for all choices as above.
	
	With $G_3$ we a.a.s.~have that
	\begin{align}
		\label{random_graph_G3} 
		\text{ any set $A$ of size at least $dn_0/2$ contains an edge of $G_3$.}
	\end{align}
	In fact, given any set $A$ of size at least $dn_0/2$, the expected number of edges of $G_3$ in $A$ is
	\[
	\binom{|A|}{2} \cdot \frac{96 T^2}{d^2 n} \ge \frac{1}{3} \cdot \frac{d^2 n_0^2}{4} \cdot \frac{96 T^2}{d^2 n} =8T^2 \frac{n_0^2}{n} \ge 2n\, ,
	\]
	where we used that $n/n_0 \le t/(1-\eps) \le 2T$.
	Therefore the probability that the set $A$ does not contain an edge of $G_3$ is at most $(1-\frac{96 T^2}{d^2 n})^{\binom{|A|}{2}} \le \exp\left(-\binom{|A|}{2} \cdot \frac{96 T^2}{d^2 n}\right) \le \exp(-2n)$ and~\eqref{random_graph_G3} follows from a union bound over the at most $2^n$ choices for $A$.
	
	Now let $M_1$ be a largest matching in $R$.
	Since $G$ is not $(\alpha,\beta)$-stable, using Lemma~\ref{lem:stable_cluster}, we conclude that $|M_1| \ge (\alpha + 2d)t$.
	At this point, for the sake of clarity, we split our proof into two cases -- $0 < \alpha < 1/3 - d/3$ and $1/3 - d/3 \le \alpha \le 1/3$ -- although some steps will be the same.
	The first case is indeed much easier, as we do not need to cover all the graph with triangles, while in the second case we are looking for a spanning structure and we want to find $\lfloor n/3 \rfloor$ pairwise vertex-disjoint triangles.
	
	\textbf{Case $0 < \alpha < 1/3 - d/3$.}
	As $M_1$ is a largest matching in $R$, the set $V(R) \setminus V(M_1)$ is independent and only one endpoint of each edge of $M_1$ can be adjacent to more than one vertex from $V(R) \setminus V(M_1)$.
	Therefore, we can greedily pick a second matching $M_2$ such that each edge of $M_2$ contains a vertex of $V(R) \setminus V(M_1)$ and a vertex of $V(M_1)$, and $M_2$ covers at least $\min\{|V(R) \setminus V(M_1)|,\delta(R) \}$ vertices of $V(R) \setminus V(M_1)$.
	The two matchings $M_1$ and $M_2$ together cover a subset $V(M_1 \cup M_2) \subset V(R)$ of 
	\[2|M_1| + \min\{|V(R) \setminus V(M_1)|,\delta(R) \} \ge \min\{t,(3\alpha+2d)t\} \ge (3 \alpha + d)t\]
	vertices, and we can extract a collection of $|M_2|$ vertex-disjoint cherries and a disjoint matching that cover such vertices.	
	This gives a subgraph $R'\subseteq R$ consisting of cherries and a matching such that for all edges $ij \in E(R')$ the pair $(V_i,V_j)$ is $\eps$-regular of density at least $d$ in $G'$, and therefore in $G$ as well.
	We denote by $\cJ \subset [t]$ the indices of the clusters $V_i$ of the cherries and the matching edges in $R'$ and we observe from above that $|\cJ| \ge (3 \alpha + d)t$.
	We add to $V_0$ all the vertices of $G$ that are in the clusters $V_j$ for $j \notin \cJ$.
	
	Then we make all pairs associated with the edges of $R'$ super-regular.
	Given a pair $(A,B)$, by Lemma~\ref{lem:MDL}, all but at most $\eps n_0$ vertices of $A$ (resp. $B$) have degree at least $(d-\eps)n_0$ to $B$ (resp. $A$).
	For every such pair we remove these vertices from $A$ and $B$, and remove additional vertices to ensure all clusters have the same size.
	As $R'$ only contains vertex-disjoint cherries and a disjoint matching, we can achieve that by removing a total of at most $2 \eps n_0$ vertices from each cluster.
	We add all the removed vertices to $V_0$.
	Observe that afterwards all the pairs $(A,B)$ associated with the edges of $R'$ are $(2\eps,d-3\eps)$-super-regular, because every vertex $a \in A$ has degree at least $(d-\eps)n_0 - 2\eps n_0 \ge (d-3\eps)|B|$ into $B$, and every vertex $b \in B$ has degree at least $(d-3\eps)|A|$ into $A$.
	
	Recall that for a later application of Lemma~\ref{lem:tripartite2} we need that for each cherry the sizes of the leaf-clusters are smaller than the size of the centre-cluster.
	Thus for each cherry $ijk$ of $R'$, with $j$ being the centre, we additionally remove $\delta |V_j| \le \delta n_0$ vertices from the leaves $V_i$ and $V_k$, and add them to $V_0$.
	We have $|V_i|=|V_k|=(1-\delta)|V_j|$ that implies $|V_i|=|V_k| \ge (1-\delta_0)|V_j|$, as $\delta_0 > \delta$.
	We have that all edges of $R'$  still give $(2\eps,d-3\eps-\delta)$-super-regular pairs.
	Moreover 
	\[\left| \bigcup_{j \in \cJ} V_j \right| \ge (1-2\eps-\delta) n_0 |\cJ| \ge (1-\eps)(1-2\eps-\delta)(3\alpha+d) n \ge 3 \alpha n\, .\]
	We can assume (by moving only a few additional vertices to $V_0$ that do not harm the bounds above) that for all cherries and matching edges in $R'$ the number of vertices in the clusters together is divisible by three.
	
	For each such super-regular cherry $ijk$ of $R'$, after revealing $G_1[V_i \cup V_j \cup V_k]$ we find by Lemma~\ref{lem:tripartite2} a.a.s.~a triangle factor covering all the vertices in $V_i \cup V_j \cup V_k$.
	Similarly for any matching edge $ij$ of $R'$, after revealing $G_1[V_i \cup V_j]$ we find by Lemma~\ref{lem:bipartite} a.a.s.~ a triangle factor covering all the vertices in $V_i \cup V_j$.
	Note that we apply Lemma~\ref{lem:tripartite2} and Lemma~\ref{lem:bipartite} only constantly many times and thus a.a.s.~we get a triangle factor in all such applications. 
	Let $\cT$ be the union of all such triangle factors.
	Then $\cT$ covers $\left|\bigcup_{j \in \cJ} V_j \right| \ge 3 \alpha n$ vertices and gives at least $\alpha n=\min \{\alpha n, \lfloor n/3 \rfloor\}$ pairwise vertex-disjoint triangles in $G \cup G(n,p)$.
	
	\textbf{Case $1/3 - d/3 \le \alpha \le 1/3$.}
	As discussed in the overview, here we cannot directly apply Lemma~\ref{lem:tripartite2} and Lemma~\ref{lem:bipartite} as in the case $0 < \alpha < 1/3-d/3$, but we need additional steps.
	However even with a lower minimum degree, we will cover all vertices of $G$ and find $\lfloor n/3 \rfloor$ pairwise vertex-disjoint triangles.
	Recall that $M_1$ is a largest matching and that $|M_1| \ge (\alpha+2d)t$.
	Then the set $V(R) \setminus V(M_1)$ is independent, has size
	\[|V(R) \setminus V(M_1)| = t - 2|M_1| \le (1-2\alpha-4d)t \le (\alpha-3d)t\]
	and only one endpoint of each edge of $M_1$ can be adjacent to more than one vertex from $V(R) \setminus V(M_1)$.
	Given that $\delta(R) \ge (\alpha -2d)t$, we can greedily pick a second matching $M_2$ such that each edge of $M_2$ contains a vertex of $V(R) \setminus V(M_1)$ and a vertex of $V(M_1)$, and $M_2$ covers the remaining vertices $V(R) \setminus V(M_1)$ completely.
	Therefore, the two matchings $M_1$ and $M_2$ together cover the vertex set $V(R)$ and we can extract a collection of $\ell = |M_2| \le (\alpha-3d)t$ vertex-disjoint cherries and a disjoint matching that cover $V(R)$.
	
	This gives a spanning subgraph $R'\subseteq R$ on vertex set $[t]$ containing $\ell \le (\alpha-3d)t$ cherries and a matching of size $(t-3\ell)/2 \ge (1-3 \alpha + 9d)t/2 \ge 9dt/2$ such that for all edges $ij \in E(R')$ the pair $(V_i,V_j)$ is $\eps$-regular of density at least $d$ in $G'$, and therefore in $G$ as well.
	We denote by $\cI \subseteq [t]$ the indices of the clusters $V_i$ that are not the centre of a cherry in $R'$.
	As above, with Lemma~\ref{lem:MDL}, we can make the pairs associated with the edges of $R'$ $(2\eps,d-3\eps)$-super-regular, while keeping the clusters all of the same size.
	For this we have to remove at most $t 2 \eps n_0 \le t 2 \eps n/t = 2 \eps n$ vertices, which we add to $V_0$.
	Next, as for a later application of Lemma~\ref{lem:tripartite2} we need that for each cherry the sizes of the leaf-clusters are smaller than the size of the centre-cluster, we remove for each $i \in \cI$ additionally $\delta |V_i| \le \delta n_0 \le \delta n/t$ vertices from $V_i$ and add them to $V_0$.
	Note that we remove vertices from the clusters of matching edges as well, although this is not necessary.
	We then get $|V_0| \le \eps n + 2 \eps n + t \delta n/t \le 2 \delta n$.
	We can assume (by moving only a few additional vertices to $V_0$ that do not harm the bounds above) that for all cherries and matching edges in $R'$ the number of vertices in the clusters together is divisible by three. 
	By removing $n \pmod 3 \in \{0,1,2\}$ vertices from $V_0$ we also have $|V_0| \equiv 0 \pmod 3$; note that this only happens when $n$ is not divisible by $3$ and we can discard these vertices.	
	
	\textbf{Covering $V_0$ with triangles.}
	We now want to cover the exceptional vertices in $V_0$ by triangles.
	It would be easy to do this greedily by just using~\eqref{random_graph_G3}, but it might happen that afterwards in many of the cherries the number of vertices is not divisible by three or that the centre cluster gets too small.
	To avoid both these issues, we will cover $V_0$ while using the same number of vertices from clusters that are together in a cherry or matching edge.
	For this we will always cover three vertices at a time and combine~\eqref{random_graph_G2} with~\eqref{random_graph_G3} to find additional triangles.
	Observe that $|V_0 \cup \bigcup_{i \not\in \cI } V_i| \le 2 \delta n + \ell n/t \le (\alpha-\gamma)n - 2dn$ and, therefore, any $v \in V_0$ has at least $2 dn$ neighbours in $\bigcup_{i \in \cI} V_i$.
	
	Assume we have already covered $V' \subseteq V_0$ vertices of $V_0$ using at most $5|V'|$ triangles in total.
	Let $W'$ be the set of vertices from $\bigcup_{i \in \cI} V_i$ used for the triangles covering $V'$ and note that $|W'| \le 30 \delta n$.
	Then let $\cI' \subseteq \cI$ be the set of indices of clusters $V_i$ with $i \in \cI$ which intersect $W'$ in at least $\sqrt{\delta}n_0$ vertices and note that $|\cI'| \le |W'|/(\sqrt{\delta}n_0) \le 30 \sqrt{\delta} t/ (1-\eps) \le 40 \sqrt{\delta} t \le dt/4$.
	Moreover, notice that as for each $v \in V_0$ we have $\deg_G(v,\bigcup_{i \in \cI} V_i) \ge 2dn$, there are at least $dt$ indices $i \in \cI$ such that $v$ has at least $dn_0$ neighbours in $V_i$.
	In particular, as $|\cI'| \le dt/4$ and $t \ge 10/d$, there are at least $dt-|\cI'| \ge 3dt/4 \ge 7$ indices $i \in \cI \setminus \cI'$ such that $v$ has at least $dn_0$ neighbours in $V_i$.
	Therefore we can pick three vertices $v_1,v_2,v_3 \in V_0 \setminus V'$ and three indices $i_1,i_2,i_3$ in $\cI \setminus \cI'$ such that $v_j$ has $dn_0$ neighbours in $V_{i_j}$ for $j=1,2,3$ and the clusters $V_{i_1},V_{i_2},V_{i_3}$ belong to pairwise different cherries or matching edges.
	For $j=1,2,3$ with~\eqref{random_graph_G3} we find an edge $e_j$ in $G_3[N(v_j,V_{i_j}) \setminus W']$ and we cover the three vertices with triangles.
	It is easy to show that we can find at most $10$ additional triangles with the help of~\eqref{random_graph_G2} and~\eqref{random_graph_G3}, in such a way that, overall, for each cherry and matching edge, we use the same number of vertices from each of their clusters; in particular, the number of vertices used from each cherry and matching edge is divisible by three.
	The clusters $V_{i_1},V_{i_2},V_{i_3}$ can belong to three cherries, two cherries and one matching edge, one cherry and two matching edges, or three matching edges. 
	We give details in the case where they are all leaves of (different) cherries, and we refer to Figure~\ref{fig:triangles} for the other three cases.
	With~\eqref{random_graph_G2} we find four triangles: two with a vertex in each of the other cluster of the cherry containing $V_{i_1}$ and the third vertex in one of the other clusters of the cherry containing $V_{i_2}$, and other two triangles with one vertex in each of the other cluster of the cherry containing $V_{i_3}$ and the third vertex in the remaining cluster of the cherry containing $V_{i_2}$.
	When a $V_{i_j}$ belongs to a matching edge of $R'$, we first find with~\eqref{random_graph_G3} two triangles inside this matching edge each with one vertex in the cluster $V_{i_j}$ and the other two vertices in the other cluster of the matching edge, then we proceed as before (see Figure~\ref{fig:triangles}).
	Note that we cover three vertices of $V_0$ using at most $13$ triangles, and thus to cover $V' \subseteq V_0$ we use at most $13 |V'|/3 \le 5 |V'|$, as claimed above.
	Therefore we can repeat this procedure until $V'=V_0$.
	
	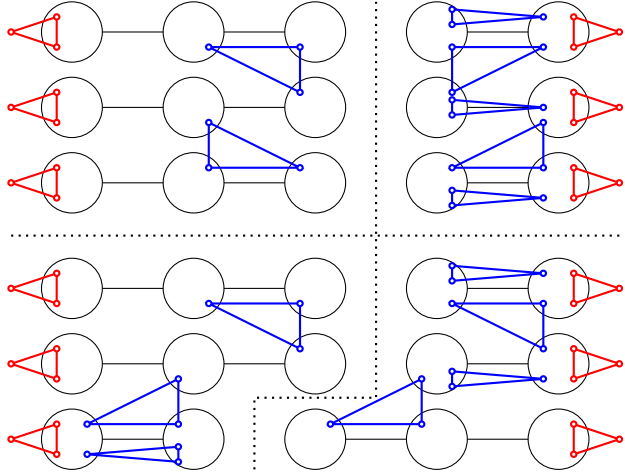
\begin{figure}[hbt]
		\begin{tikzpicture}[scale=0.2]
			\tikzstyle{every node}=[circle, draw,fill=white,
			inner sep=0pt, minimum width=2pt]
			
			\foreach \s in {4,12,20,28,36}
			\foreach \t in {4,9,14,21,26,31}
			\draw (\s,\t) circle (2cm);		
			
			\foreach \s in {6,30}
			\foreach \t in {4,9,14,21,26,31}
			\draw (\s,\t) -- (\s+4,\t);
			
			\foreach \t in {9,14,21,26,31}
			\draw (14,\t) -- (18,\t);
			
			\draw (22,4) -- (26,4);
			
			\foreach \s/\t in {13/30,13/13,29/13}
			\draw[blue,thick] (\s,\t) node {} -- (\s+6,\t) node {} -- (\s+6,\t-3) node {} -- cycle;
			
			\foreach \s/\t in {13/25}
			\draw[blue,thick] (\s,\t) node {} -- (\s,\t-3) node {} -- (\s+6,\t-3) node {} -- cycle;
			
			\foreach \s/\t in {5/5,29/22,21/5}
			\draw[blue,thick] (\s,\t) node {} -- (\s+6,\t) node {} -- (\s+6,\t+3) node {} -- cycle;
			
			\draw[blue,thick] (29,30) node {} -- (29,27) node {} -- (35,30) node {} -- cycle;
			
			\foreach \t in {8,15,20,26,32}
			\draw[blue,thick] (29,\t-0.5) node {} -- (29,\t+0.5) node {} -- (35,\t) node {} -- cycle;
			
			\draw[blue,thick] (5,3) node {} -- (11,2.5) node {} -- (11,3.5) node {} -- cycle;
			
			\foreach \t in {4,9,14,21,26,31}
			\draw[red,thick] (37,\t-1) node {} -- (37,\t+1) node {} -- (40,\t) node {} -- cycle;
			
			\foreach \t in {4,9,14,21,26,31}
			\draw[red,thick] (3,\t-1) node {} -- (3,\t+1) node {} -- (0,\t) node {} -- cycle;
			
			\draw[thick,dotted] (16,2) -- (16,6.75) -- (24,6.75) -- (24,33);		
			\draw[thick,dotted] (0,17.5) -- (40,17.5);
		\end{tikzpicture}
		\caption{Embeddings of triangles for absorbing $V_0$ while using the same number of vertices from each cluster within a cherry or a matching edge. Each red triangle covers a vertex of $V_0$. Each blue triangle stands for two triangles with end-points in the same clusters; we only draw one for simplicity.}
		\label{fig:triangles}
	\end{figure}
	
	Let $\cT_1$ be the set of triangles we found above to cover $V_0$ and keep the divisibility condition.
	We now update the regularity partition by deleting $V(\cT_1)$ from each $V_i$ for $i \in [t]$ and note that for all cherries and matchings from $R'$ the number of vertices in the clusters together is divisible by three.
	We recall that so far we removed at most $(2 \eps + \delta + 2\sqrt{\delta}) n_0$ vertices from each cluster, where the first (resp. second, third) term bounds the number of vertices removed for making each pair super-regular (resp. for a later application of Lemma~\ref{lem:tripartite2}, for covering $V_0$).

	\textbf{Balancing the partition.}
	Now the matching edges in $R'$ are already ready for an application of Lemma~\ref{lem:bipartite} and we will not modify the corresponding clusters anymore.
	However, before an application of Lemma~\ref{lem:tripartite2} to the cherries in $R'$, we need to ensure that the ratio between their size and the size of the centre-cluster satisfies the hypotheses of the lemma.
	This is what we are going to do now.
	For $i \not\in \cI$ we denote by $W_i$ and $U_i$ the leaf-clusters of the cherry centred at $V_i$.
	Before covering the vertices of $V_0$, we had $|W_i| = |U_i| \le (1-\delta) |V_i|$ for $i \not\in \cI$, which still holds as we removed the same number of vertices from each cluster of a cherry.

	However we still need to guarantee the other inequality $|W_i|,|U_i| \ge (1-\delta_0) |V_i|$.	
	For that, we find $2m$ triangles with two vertices in $V_i$, of which one half has the third vertex in $U_i$ and the other half in $W_i$, where $m$ is the smallest integer such that 
	\begin{equation}
		\label{eq:2m}
		|U_i|-m \ge (1-\delta_0)(|V_i|-4m).
	\end{equation}
	Then after removing these $2m$ triangles, we will have precisely $(1-\delta_0)|V_i| \le |U_i| = |W_i|$.
	Observe that the inequality~\eqref{eq:2m} implies that $m \ge \frac{(1-\delta_0)|V_i|-|U_i|}{4(1-\delta_0)-1}$ and, as we chose the smallest such $m$, we get $m \le \lceil \frac{(1-\delta_0)|V_i|-|U_i|}{4(1-\delta_0)-1} \rceil$.
	Moreover as $\delta < \delta_0$, $|V_i| \le n_0$ and $|U_i| \ge (1-2\eps-\delta-2 \sqrt{\delta} ) n_0$, we have $\frac{(1-\delta_0)|V_i|-|U_i|}{4(1-\delta_0)-1} < \frac{(1-\delta_0)-(1-2\eps-\delta-2 \sqrt{\delta} )}{2}n_0 < 2 \sqrt{\delta}n_0$.
	Therefore, for $n$ (and thus $n_0$) large enough, $m \le 2 \sqrt{\delta}n_0$.
	We can find these at most $4 \sqrt{\delta} n_0$ triangles, by iteratively picking them with~\eqref{random_graph_G3} and removing the corresponding vertices from $U_i$, $W_i$, and $V_i$.
	Indeed, for any $v \in W_i \cup U_i$ we have degree into $V_i$ at least $(d-3 \eps- \delta -10 \sqrt{\delta})n_0 \ge dn_0/2$, as we started from $(2\eps,d-3\eps)$-super-regular pairs and $\delta < \delta' < 160^{-2} d^2$.
	
	Note that afterwards we still have $|U_i| \le (1-\delta)|V_i|$ as for large enough $n$ and with $\delta < \delta_0$ we have $m \le \lceil \frac{(1-\delta_0)|V_i|-|U_i|}{4(1-\delta_0)-1} \rceil \le \frac{(1-\delta)|V_i|-|U_i|}{4(1-\delta)-1}$.
	Therefore, we have $(1-\delta_0)|V_i| \le |U_i| = |W_i| \le (1-\delta) |V_i|$.
	Moreover with $d-3 \eps-\delta-10 \sqrt{\delta} \ge d/2$ and $2\eps \le \eps'$, we get that the pairs $(U_i,V_i)$ and $(W_i,V_i)$ are $(\eps',d/2)$-super-regular.	
	Let $\cT_2$ be the set of triangles we removed during this phase.
	
	\textbf{Completing the triangles.}
	Now for any $i \not\in \cI$, after revealing $G_1[V_i \cup W_i \cup U_i]$, we a.a.s.~find  a triangle factor covering the vertices of $U_i$, $W_i$, and $V_i$ by Lemma~\ref{lem:tripartite2}.
	Similarly for any matching edge $ij$ of $R'$ observe that $(V_i,V_j)$ is a $(\eps',d/2)$-super-regular pair.
	Then after revealing $G_1[V_i \cup V_j]$, we a.a.s.~find a triangle factor covering the vertices of $V_i$ and $V_j$ by Lemma~\ref{lem:bipartite}.
	Note that we apply Lemma~\ref{lem:tripartite2} and Lemma~\ref{lem:bipartite} only constantly many times and thus a.a.s.~we get a triangle factor in all such applications. 
	Let $\cT_3$ be the union of the triangle factors we obtain for each $i \not\in \cI$ and each matching edge $ij$ from $R'$.
	Then $\cT_1 \cup \cT_2 \cup \cT_3$ gives $\lfloor n/3 \rfloor$ pairwise vertex-disjoint triangles in $G \cup G(n,p)$.
\end{proof}

\section{Proof of the Sublinear Theorem}
\label{sec:fewtriangles}

As outlined in the overview, we use the following two Propositions to prove Theorem~\ref{thm:sublinear}. 

\begin{proposition}
	\label{prop:lesqrt}
	For any $0<\gamma<1/2$ there exists $C>0$ such that for any $(\log n)^3 \le m \le \sqrt{n}$ and any $n$-vertex graph $G$ with maximum degree $\Delta(G) \le \gamma n$ and minimum degree $\delta(G) \ge m$ the following holds.
	With $p \ge C /n$ there are a.a.s.~at least $m$ pairwise vertex-disjoint triangles in $G \cup G(n,p)$.
\end{proposition}

\begin{proposition}
	\label{prop:gesqrt}
	There exists $C>0$ such that for any $\sqrt{n} \le m \le n/32$ and any $n$-vertex graph $G$ with maximum degree $\Delta(G) \le n/32$ and minimum degree $\delta(G) \ge m$ the following holds.
	With $p \ge C \log n/n$ there are a.a.s.~at least $m$ pairwise vertex-disjoint triangles in $G \cup G(n,p)$.
\end{proposition}

With this at hand we can prove Theorem~\ref{thm:sublinear}.

\begin{proof}[Proof of Theorem~\ref{thm:sublinear}]
	Let $1 \le m \le n/256$ and let $G$ be an $n$-vertex graph on vertex set $V$ with minimum degree 
	$\delta(G) \ge m$.
	We let $C$ be large enough such that with $p \ge C \log n/n$ we can expose $G(n,p)$ in four rounds as $\bigcup_{i=1}^4 G_i$ with $G_i \sim G(n,C_i \log n/n)$ for $i=1,\dots,4$ such that the following hold.
	We let $C_1=1$ and observe that, by a union bound, a.a.s.~for any set of vertices $U$ of size at least $n/256$ there is at least one edge in $G_1[U]$.
	Next, we let $C_2$ be large enough such that for a set of vertices $U$ of size at least $n/2$ there are a.a.s.~at least $\log^3 n$ pairwise vertex-disjoint triangles in $G_2[U]$~\cite[Theorem~3.29]{JLR}.
	Finally, let $C_3=1$ (this is sufficient for our application of Proposition~\ref{prop:lesqrt} because we do not need the $\log n$ term) and $C_4$ be such that we can apply Proposition~\ref{prop:gesqrt} with $C_4/2$.
	We expose $G_1$ already now and assume that the described property holds, while we leave $G_2$, $G_3$, and $G_4$ until we need them.
	
	To apply one of the two propositions to a large subgraph $G'$ of $G$ we need $\Delta(G') \le v(G')/32$.
	For this let $V'$ be the set of vertices from $G$ of degree at least $n/64$.
	If $|V'|\ge m$, then we let $V''$ be any subset of $V'$ of size $m$ and we greedily find $m$ pairwise vertex-disjoint triangles in $G \cup G_1$, each containing exactly one vertex from $V''$.
	Indeed, as long as we have less than $m$ triangles there is a vertex $v \in V''$ not yet contained in a triangle.
	Then there is a set $U \subseteq N_G(v) \setminus V''$ of at least $n/64-3m \ge n/256$ vertices not covered by triangles, and we can find an edge within $G_1[U]$ that gives us a triangle containing $v$ and two vertices from $U$.
	
	Otherwise, $|V'|<m$ and we remove $V'$ from $G$ to obtain $G'=G[V \setminus V']$.
	Note that we have $v(G') = n-|V'| \ge n/2$, minimum degree $\delta(G') \ge m'= m-|V'|$, and maximum degree $\Delta(G') < n/64 \le v(G')/32$.
	If $m' < (\log n)^3$, then we a.a.s.~find $m'$ pairwise vertex-disjoint triangles within $G_2[V(G')]$.
	If $(\log v(G'))^3 \le (\log n)^3 \le m' \le \sqrt{v(G')}$, then by Proposition~\ref{prop:lesqrt} and as $\log n/n= \omega(1/v(G'))$ there are a.a.s.~at least $m'$ pairwise vertex-disjoint triangles in $G' \cup G_3[V(G')]$.
	Finally, if $\sqrt{v(G')} \le m' \le n/256 \le v(G')/32$, then by Proposition~\ref{prop:gesqrt} and as $C_4 \log n/n \ge \frac{C_4}{2} \log v(G') / v(G')$, there are a.a.s.~at least $m'$ pairwise vertex-disjoint triangles in $G' \cup G_4[V(G')]$.
	
	Now, that we found $m'$ pairwise vertex-disjoint triangles, we can greedily add triangles by using the $m-m'$ vertices from $V'$ and an edge in their neighbourhood until we have $m$ triangles.
	Analogous to above, as long as we have less than $m$ triangles, for each available vertex $v \in V'$, there is a set $U \subset N_G(v) \setminus V'$ of at least $n/256$ vertices not covered by triangles, and we find an edge within $G_1[U]$.
\end{proof}

It remains to prove Proposition~\ref{prop:lesqrt} and Proposition~\ref{prop:gesqrt}.
For Proposition~\ref{prop:lesqrt}, which deals with the cases $(\log n)^3 \le m \le \sqrt{n}$, we first need to find many large enough vertex-disjoint stars in $G$.
These can be found deterministically with Lemma~\ref{lem:manystars} below and afterwards we will show that a.a.s.~at least $m$ of them can be completed to triangles with the help of $G(n,p)$. 

For any integer $g \ge 2$, we define the \emph{star} on $g+1$ vertices as the graph with one vertex of degree $g$ (this vertex is called the \emph{centre}) and the other vertices of degree one (these vertices are called \emph{leaves}).
Given a star $K$, we denote the number of its leaves by $g_K$. 
Moreover, given a family of vertex-disjoint stars $\cK$, we denote the set of all their centre vertices by $\cK_C$ and the set of all their leaf vertices by $\cK_L$.

\begin{lemma}
	\label{lem:manystars}
	For every $0<\gamma<1/2$ and integer $s$ there exists an $\eps>0$ such that    
	for $n$ large enough and any $m$ with $2/\eps \le m \le \sqrt{n}$ the following holds.
	In every $n$-vertex graph $G$ with minimum degree $\delta(G) \ge m$ and
	maximum degree $\Delta(G) \le \gamma n$ there exists a family $\mathcal{K}$ of
	vertex-disjoint stars in $G$ such that every $K\in {\mathcal K}$ has $g_K$ leaves with
	$\eps m \le g_K \le \eps \sqrt{n}$ and
	\[\sum_{K\in{\mathcal K}} g_K^2 \ge s \eps^2 n m.\]
\end{lemma}

\begin{proof}[Proof of Lemma~\ref{lem:manystars}]
	
	Given $0<\gamma<1/2$ and an integer $s$ we let $\eps>0$ such that $\eps\le 1/(6s)$ and $\eps < 1/2- s \gamma$.
	Moreover, we let $n$ be large enough for our calculations and, for simplicity, assume that $\eps \sqrt{n}$ is an integer.
	Then let $2/\eps \le m \le \sqrt{n}$ and $G$ be an $n$-vertex graph on vertex set $V$ with $\delta(G) \ge m$ and $\Delta(G) \le \gamma n$.
	
	Let $\mathcal{K}$ be a family of vertex-disjoint stars in $G$ with $\eps m \le g_K \le \eps \sqrt{n}$ for all $K \in \cK$, that maximizes the sum
	\begin{equation}
		\label{eq:sum}
		\sum_{K \in \cK} g_K^2
	\end{equation}
	among all such families. Note that each star in $\cK$ has at least $2$ leaves because $g_K \ge \eps m$ and $m \ge 2/\eps$.
	
	If the sum in~\eqref{eq:sum} is bigger than $s \eps^2 n m$ we are done.
	So we assume the family $\mathcal{K}$ satisfies
	\begin{equation}
		\label{eq:smallsum}
		\sum_{K \in \mathcal{K}} g_K^2 < s \eps^2 n m.
	\end{equation}
	We are going to prove that then there exists a vertex of degree larger than $\gamma n$, contradicting our assumption on the maximum degree.
	
	For this we split $\mathcal{K}$ into two subfamilies
	\begin{align*}
		\cM=\left\{K \in \mathcal{K}:\eps m \le g_K <  \eps \sqrt{n}\right \} \quad \text{and} \quad
		\cH=\left\{K \in \mathcal{K}:g_K=  \eps \sqrt{n}\right  \}
	\end{align*}
	and we let $R$ be the set of vertices not covered by the stars in $\cK$, that is $R=V(G) \setminus (\cH_C \cup \cH_L \cup \cM_C \cup \cM_L)$, where $\cH_C$, $\cH_L$, $\cM_C$, and $\cM_L$ are obtained from $\cM$ and $\cH$ as defined above.
	
	For all stars $K \in \cH$ we have $g_K^2=\eps^2 n$.
	From~\eqref{eq:smallsum} we get that the subfamily $\cH$ contains at most $s m$ stars and hence
	\begin{gather}
		\label{eq:H_c_l}
		|\cH_C| \le sm \le s \sqrt{n} \quad  \text{and} \quad
		|\cH_L| = |\cH_C| \eps \sqrt{n} \le s m \eps \sqrt{n} \le s \eps n,
	\end{gather}
	because $m \le \sqrt{n}$.
	
	As each star in $\cM$ has at least $\eps m$ leaves we have $\sum_{K \in \cM} g_K \ge |\cM| \eps m$.
	Using the Cauchy-Schwarz inequality, we then get
	\[\left(\sum_{K \in \cM} g_K\right)^2 \le \left(\sum_{K \in \cM} g_K^2 \right) |\cM| \leBy{\eqref{eq:smallsum}} s \eps^2 nm |\cM| \le \left(\sum_{K \in \cM} g_K\right) s \eps n, \]
	which implies 
	\begin{align}
		\label{eq:M_l}
		|\cM_L| = \sum_{K \in \cM} g_K   \le s \eps n.
	\end{align}
	Therefore, $|\cM_C| \le |\cM_L|/2 \le s \eps n/2$ since each star has at least $2$ leaves.
	
	These bounds on $\cM_L$ and $\cM_C$ together with~\eqref{eq:H_c_l} immediately imply that $|R| \ge (1-3 s \eps) n \ge n/2$.
	We are going to show that there are many edges between $R$ and $\cH_C$ and from that we derive the existence of a high degree vertex, giving the desired contradiction.
	
	A vertex in $R$ cannot have at least $\eps m$ neighbours inside $R$, because otherwise we could create a new star and increase the sum in~\eqref{eq:sum}.
	Therefore, $e(R) < |R| \eps m /2$.
	We also have $e(R,\cM_C)=0$ since otherwise we could add an edge to one of the existing stars in $\cK$ increasing the sum in~\eqref{eq:sum} (recall that stars in $\cM$ have less than $\eps \sqrt{n}$ leaves).
	
	Given a leaf $v \in \cM_L$ that belongs to a star $K$ with $g_K$ leaves, we must have $\deg(v,R) < g_K+1$. 
	Otherwise, we could take $V' \subset N_R(v)$ of size $|V'|=g_k+1  \le \eps \sqrt{n}$ and create a new family of vertex-disjoint stars, given by $\cK \setminus \{K\}$ and the star on $v \cup V'$, to increase the sum in~\eqref{eq:sum}.
	Therefore,
	\[e(R,\cM_L) \le \sum_{K \in \cM} g_K (g_K+1) \leq \sum_{K \in \cK} g_K^2 + \sum_{K \in \cM} g_K \leBy{\eqref{eq:smallsum},\eqref{eq:M_l}} s \eps^2 n m + s \eps  n.\]
	Similarly, given $v \in \cH_L$, we must have $\deg(v,R) < \eps \sqrt{n}$.
	Otherwise, we could take $V' \subset N_R(v)$ of size $|V'|=\eps \sqrt{n}$ and create a new family of vertex-disjoint stars, given by $\cK \setminus \{K\}$, the star $K \setminus \left\{v\right\}$, and the star on $v \cup V'$, to increase the sum in~\eqref{eq:sum}.
	Therefore, $e(R,\cH_L) \le |\cH_L| \eps \sqrt{n} \le s \eps^2 n m$ by~\eqref{eq:H_c_l}.
	
	On the other hand, $\delta(G) \ge m$ implies $e(R,V) \ge m |R|$, where the edges inside of $R$ are counted twice.
	Then we can lower bound the number of edges between $R$ and $\cH_C$ by
	\begin{align*}
		e(R,\cH_C) & \ge m |R|-2 e(R)-e(R,\cM_C)-e(R,\cM_L)-e(R,\cH_L) \\
		&\ge m |R| - \eps m |R|- 0 - 2 s \eps^2 n m - s \eps n \\
		&\ge m |R| - \eps m |R| - 4 s \eps^2 m |R| - s \eps^2 m |R| \\
		&\ge (1-\eps - 5 s \eps^2) |R| m \ge  |R|(1-2 \eps)m.
	\end{align*}
	where we used the bounds on $e(R)$, $e(R,\cM_C)$, $e(R,\cM_L)$, and $e(R,\cH_L)$ we found above, together with $|R| \ge n/2$, $\eps m \ge 2$ and the choice of $\eps < 1/(6s)$.
	In particular, as $|\cH_C| \le sm$ and using $\eps < 1/2- s \gamma$, and $|R| > n/2$, there exists a vertex $v \in \cH_C$ of degree
	\[\deg(v) \ge \deg(v,R) \ge \frac{|R| (1-2 \eps)m}{sm} \ge |R| 2 \gamma > \gamma n.\]
	This contradicts the maximum degree of $G$.
\end{proof}

\begin{proof}[Proof of Proposition~\ref{prop:lesqrt}]
	Let $n$ be sufficiently large for the following arguments.
	With $0<\gamma<1/2$ let $G$ be an $n$-vertex graph with maximum degree $\Delta(G) \le \gamma n$ and minimum degree $\delta(G) \ge m$.
	With $(\log n)^3 \le m \le \sqrt{n}$, we first find many vertex-disjoint stars in $G$ and then complete at least $m$ of them to triangles with the help of $G(n,p)$.
	We apply Lemma~\ref{lem:manystars} with $\gamma$ and $s=8$ to get $0<\eps<1/2$ and, as $n$ is large enough and $m \ge 2/\eps$, we get a family $\cK$ of vertex-disjoint stars on $V(G)$ such that $\eps m \le g_K \le \eps \sqrt{n}$ for $K \in \cK$ and $\sum_{K \in \cK} g_K^2 \ge 8 \eps^2 n m$.
	
	As we have stars of different sizes, we split $\cK$ into $t=\lceil \log (\sqrt{n}/m)/\log 2 \rceil+1$ subfamilies
	\[\cK_i = \{ K \in \cK \,\colon 2^{i-1} \eps m \le g_K < 2^i \eps m \}, \quad 1\le i \le t,\]
	and set $k_i =|\cK_i|$. 
	
	By deleting leaves, we may assume that all stars in $\cK_i$ have exactly $\lceil 2^{i-1} \eps m \rceil$ leaves.
	Denote by $\cI$ the set of indices $i \in [t]$ such that $k_i \left( 2^{i-1} \eps m \right)^2   \ge \eps^2 n m /t$.
	Next we prove that 
	$\sum_{i \in \cI} k_i \left( 2^{i-1} \eps m \right)^2 \ge \eps^2 nm.$
	
	Observe first that $\sum_{i \not\in \cI} k_i \left( 2^{i-1} \eps m \right)^2 \le t (\eps^2 nm/t)=\eps^2 nm$.
	It follows that
	\begin{eqnarray*}
		\sum_{i \in \cI} k_i \left( 2^{i-1} \eps m \right)^2 &=& \frac{1}{4} \sum_{i \in \cI} k_i \left( 2^i \eps m \right)^2 =
		\frac{1}{4} \sum_{i=1}^t |\cK_i| \left( 2^i \eps m \right)^2 
		-\sum_{i \not\in \cI} k_i \left( 2^{i-1} \eps m \right)^2\\
		&\ge&\frac{1}{4} \sum_{i=1}^t \sum_{K\in\cK_i} g_K^2 - \eps^2 nm\ge 2 \eps^2 nm - \eps^2 nm \ge \eps^2 nm.
	\end{eqnarray*}
	
	Now we reveal random edges on $V(G)$ with probability $p \ge C /n$ where $C$ is large enough for the Chernoff bounds and inequalities below.
	We shall show that this allows us to find at least $m$ triangles~a.a.s..
	Indeed, for each $i \in \cI$, we find many pairwise vertex-disjoint triangles in $\cK_i$ using random edges.
	\begin{claim}
		\label{claim:StarsTriangles}
		For any $i \in \cI$, after revealing edges of $G(n,p)$ with $p \ge C /n$ we have with probability at least $1-1/n$ at least $k_i (2^{i-1}m)^2/ n$ pairwise vertex-disjoint triangles within $(G \cup G(n,p))[\cup_{K \in \cK_i} V(K)]$.
	\end{claim}
	Having this claim and since $|\cI| \le t =o (n)$, with a union bound over $i \in \cI$, there are a.a.s.~at least
	\[\sum_{i \in \cI} k_i (2^{i-1}m)^2/ n = \frac{1}{\eps ^2 n} \sum_{i \in \cI} k_i \left( 2^{i-1} \eps m \right)^2 \ge \frac{\eps^2 nm}{\eps^2 n} \ge m\]
	pairwise vertex-disjoint triangles in $G\cup G(n,p)$.
	It remains to prove Claim \ref{claim:StarsTriangles}.
	
	\begin{claimproof}[Proof of Claim~\ref{claim:StarsTriangles}]
		Fix $i \in \cI$ and let $k=k_i$ and $g = \lceil 2^{i-1} \eps m \rceil$.
		We reveal random edges with probability $p$ within each set of leaves of the $k$ stars in $\cK_i$.
		We recall that these $k$ sets are pairwise disjoint and each has size $g$.
		Let $X_j$ be the indicator variable of the event that the $j$-th of these sets contains at least one edge for $1 \le j \le k$, and set $X = \sum_{j=1}^kX_i$.
		Then $\PP[X_j=1]=1-(1-p)^{\binom{g}{2}}$ and $\EE[X] = k \left( 1-(1-p)^{\binom{g}{2}} \right)$.
		We have that $\EE[X] \ge 2 k g^2/(\eps^2 n)$.
		Indeed,
		\begin{align*}
			k \left( 1-(1-p)^{\binom{g}{2}} \right)  \ge 2 k g^2/(\eps^2 n) \quad
			\Leftrightarrow  \quad 1-2g^2/(\eps^2 n) \ge \left( 1- \frac{C }{n}\right)^{\binom{g}{2}},
		\end{align*}
		and the later holds for large enough $C$ and $n$ using the inequality $1-x\le e^{-x}\le 1-\frac{x}{2}$ valid for $x<3/2$.
		
		From Chernoff's inequality (Lemma~\ref{lem:chernoff}) and from the fact that $k g^2/(\eps^2 n) \ge m/t$ by the definition of $\cI$, it follows that with probability at most
		\[2 \exp \left( -\frac{1}{6} \frac{k g^2}{\eps^2 n} \right) \le 2 \exp\left( - \frac{1}{6} \frac{m}{t} \right) \le \frac{1}{n}\]
		there are less than $kg^2/(\eps^2 n)$ triangles, where the last inequality holds as $t \le \log n$, $m \ge (\log n)^3$ and $n$ is large enough.
	\end{claimproof}
\end{proof}

\begin{proof}[Proof of Proposition~\ref{prop:gesqrt}]
	Let $G$ be an $n$-vertex graph with maximum degree $\Delta(G) \le n/32$ and minimum degree $\delta(G) \ge m$.
	With $\sqrt{n} \le m \le n/32$ we can not hope to find sufficiently many large enough vertex-disjoint stars (as we did in the proof of Proposition~\ref{prop:lesqrt}).
	Instead we apply a greedy strategy, using that a.a.s.~every vertex of $G$ has random edges in its neighbourhood.
	We can greedily obtain a spanning bipartite subgraph $G' \subseteq G$ of minimum degree $\delta(G') \ge m/2$ by taking a partition of $V(G)$ into sets $A$ and $B$ such that $e_G(A,B)$ is maximised and letting $G'=G[A,B]$.
	Indeed, a vertex of degree less than $m/2$ can be moved to the other class to increase $e_G(A,B)$.
	W.l.o.g.~we assume $|B|\ge n/2\ge |A|$.
	Moreover, we have $|A| \ge 8m$, as otherwise with $e(A,B) \ge nm/4$ there is a vertex of degree larger than $n/32$, a contradiction.
	
	\begin{claim}	\label{claim:A'B'}
		For every $A' \subseteq A$, $B' \subseteq B$ with $|A'| < 2m$, $|B'| \ge n/4$ we have $e(A \setminus A',B') \ge n m/16$.
	\end{claim}
	\begin{claimproof}
		If $e(A \setminus A',B') < n m/16$, it follows from $e(A,B') \ge |B'| m /2\ge nm/8$ that we have $e(A',B') \ge n m /16$. Since $|A'| < 2m$, there must be a vertex of degree at least $n/32$ in $A'$, a contradiction.
	\end{claimproof}
	
	From this claim it follows that there are many vertices of high degree in $A \setminus A'$.
	\begin{claim}\label{claim:largeA*}
		Suppose that $A' \subseteq A$, $B' \subseteq B$ with $|A'| < 2m$, $|B'| \ge n/4$. 
		Let $ A^* = \{ v \in A \setminus A' \,\colon\, \deg (v,B') \ge m/16 \}$. Then $|A^*| \ge m.$
	\end{claim}
	\begin{claimproof}
		We have $|A^*| n/32 + |A| m/16 \ge e(A^*,B') +e(A \setminus (A'\cup A^*),B') = e(A \setminus A',B') \ge nm/16$, where the last inequality uses Claim~\ref{claim:A'B'}.
		Since $|A| \le n/2$, we get
		\[ |A^*| \ge \frac{nm/16-nm/32}{n/32} = m.\]
	\end{claimproof}
	
	Let $s=\lceil \frac{2n}{m} \rceil$ and $t= \lceil \frac{m^2}{2n} \rceil$.
	We will now iteratively construct our $m$ triangles in $t$ rounds of $s$ triangles each.
	In each round we will reveal $G(n,q)$ with $q = \frac{C \log n}{m^{2}}$, where $C$ is large enough for the Chernoff bound below.
	For the start we set $A'=B_0=\emptyset$.
	
	Let $i=1,\dots,t$, suppose that before the $i$-th round we have 
	\[ 
	|A'| = (i-1)s \le \left( \Big \lceil \frac{m^2}{2n} \Big \rceil -1 \right) \Big \lceil \frac{2n}{m} \Big \rceil \le \frac{m^2}{2n} \left( \frac{2n}{m} +1 \right)< 2m
	\]
	and $|B_0|=(i-1)(2s) < 3m$, and note this is true for $i=1$.
	In the $i$-th round we pick vertices $v_1,\dots,v_s \in A \setminus A'$ and pairwise disjoint sets $B_1,\dots,B_s \subseteq B \setminus B_0$, each of size $\lceil m/16 \rceil$, such that $B_j \subset N_{G'}(v_j)$ for each $j=1,\dots,s$.
	We can do this greedily, where for $j=1,\dots,s$ we set $B'=B \setminus (B_0 \cup B_1 \cup \dots \cup B_{j-1})$ and apply Claim~\ref{claim:largeA*} to obtain a vertex $v_j \in A \setminus A'$ together with a set $B_j \subseteq B'$ of $\lceil m/16 \rceil$ neighbours of $v_j$.
	We can do this as $|A'| < 2m$ and $|B'| \ge n/2 - s \lceil m/16 \rceil - |B_0| \ge n/4$ as $m\le n/32$.
	
	Now we reveal additional edges at random with probability $q$.
	Then with probability at least $1-1/n^2$ we have at least one edge in each set $B_1, \dots, B_s$.
	Indeed the probability that there is no edge in a set $B_i$ is at most $(1-q)^{\binom{|B_i|}{2}} \le \exp(-C \tfrac{\log n}{m} \binom{\lceil m/16 \rceil}{2}) \le  n^{-3}$ as $C$ is large enough.
	Therefore the probability that there is a set without any edge is at most $s n^{-3} \le n^{-2}$ by a union bound.
	We fix an arbitrary edge from each $B_i$ and together with $v_1,\dots,v_s$ this gives us $s$ triangles.
	We add the vertices $v_1,\dots,v_s$ to $A'$ and the vertices of the edges that we chose to $B_0$.
	Notice that $|A'|=is$ and $|B'|=i(2s)$, as required at the beginning of next round.
	
	We can repeat the above $t$ times because with $m \ge \sqrt{n}$ we get $t q \le \frac{C \log n}{n} = p$.
	By a union bound over the $t=\lceil \frac{m^2}{2n} \rceil \le n$ rounds, we get that we succeed a.a.s.~and find $t s \ge m$ triangles.
	
\end{proof}

\section{Proof of the auxiliary lemmas}
\label{sec:auxiliary}

In this section we prove Lemmas~\ref{lem:tripartite},~\ref{lem:tripartite2}, and~\ref{lem:bipartite}.
For each of them, we first give an overview of the strategy and then a full proof.

\subsection{Proof of Lemmas~\ref{lem:tripartite} and~\ref{lem:tripartite2}.}
We describe the general setup of both Lemmas~\ref{lem:tripartite} and~\ref{lem:tripartite2}.
Let $G$ be a graph on $U \cup V \cup W$ with $(V,U)$ and $(V,W)$ being super-regular with respect to $G$.
We will find all/most triangles, respectively, with one vertex in each of the sets $U$, $V$, and $W$, with the edges between $U$ and $W$ coming from the random graph.
To find these edges we consider a random matching $M$ in $G(U,W,p)$ such that each matching edge is contained in many triangles with the third vertex from $V$.
As we later want to match edges from $M$ to vertices from $V$, in order to get triangles, we consider the following bipartite auxiliary graph.
Given a matching $M$ between $U$ and $W$ the vertex set of the graph $H_G(M,V)$ consists of $M$ and $V$ and there is an edge between $m \in M$ and $v \in V$ if and only if the vertices of $m$ are incident to $v$ in $G$.
The lemma below states that with $ p \ge C/n$ we can a.a.s.~find a large matching $M$ such that additionally $(M,V)$ gives a super-regular pair in $H_G(M,V)$.
Observe that a matching within $H_G(M,V)$ induces pairwise vertex-disjoint triangles in $G \cup G(U,W,p)$.

\begin{lemma}
	\label{lem:trireg}
	For any $0 < d,\delta,\eps' < 1$ with $2\delta \le d$ there exist $\eps,C>0$ such that the following holds.
	Let $G$ be a graph on $U \cup V \cup W$, with $|V|=n$ and $(1-1/2)n \le |U|=|W| \le (1 + 1/2) n$, such that $(V,U)$ and $(V,W)$ are $(\eps,d)$-super-regular pairs with respect to $G$.
	Further, let $G(U,W,p)$ be a random graph with $p \ge C/n$.
	Then a.a.s.~there exists a matching $M \subseteq G(U,W,p)$ of size $|M| = (1-\delta)|W|$ such that the pair $(M,V)$ is $(\eps',d^3/32)$-super-regular with respect to the auxiliary graph $H_G(M,V)$.
\end{lemma}

We will prove this lemma at the end of this section.
For Lemma~\ref{lem:tripartite2}, we will use the minimum degree condition and $G(V,p)$ to find additional triangles covering the remaining vertices from $U$ and $W$ that are not covered by $M$. 
We now proceed to the details of this proof.

\begin{proof}[Proof of Lemma~\ref{lem:tripartite2}]
	Given $0 < \delta' \le d <1$, let $\eps'>0$ be given by Lemma~\ref{lem:blow-up} on input $d^3/64$ and let $0<\delta< \delta_0 \le \min \{ \delta' , d^3/22 \}$.
	Furthermore, let $C \ge 8 \delta^{-2}$, let $0<\eps<\delta/4$ be given by Lemma~\ref{lem:trireg} on input $d$, $\tfrac{\delta}{3(1-\delta)}$ (in place of $\delta$), and $\eps'/2$ and let $p \ge C/n$.
	
	Suppose $U$, $V$, $W$ are disjoint sets of size $|V|=n$ and $(1-\delta_0)n \le |U|=|W| \le (1-\delta)n$ with $|V|+|U|+|W| \equiv 0 \pmod 3$, and $G$ is a graph with vertex set $U \cup V \cup W$ such that the pairs $(V,U)$ and $(V,W)$ are $(\eps,d)$-super-regular with respect to $G$.
	Let $\delta_1$ be such that $|U|=|W|=(1-\delta_1)n$ and observe that $\delta \le \delta_1 \le \delta_0$.
	We reveal random edges $G_1 \sim G(U,W,p)$ and $G_2 \sim G(V,p)$ and we have that a.a.s.~any set of size at least $\delta n $ in $V$ contains an edge of $G_2$.
	Indeed, fixed a set of size at least $\delta n$, the probability that it does not contain an edge of $G_2$ is at most $(1-p)^{\binom{\delta n}{2}} \le \exp(-p\binom{\delta n}{2}) \le \exp(-2n)$ since $C \ge 8 \delta^{-2}$, and we conclude by a union bound over the at most $2^n$ choices of such set.
	Then we apply Lemma~\ref{lem:trireg} with $G_1$ to obtain a matching $M \subseteq G_1$ of size $|M| = \left(1-\tfrac{\delta}{3(1-\delta)}\right)|W|=\left(1-\tfrac{\delta}{3(1-\delta)}\right)(1-\delta_1)n$ such that the pair $(M,V)$ is $(\eps'/2,d^3/32)$-super-regular with respect to $H_G(M,V)$.
	As for $x \in (0,1)$ the function $x \rightarrow x/(1-x)$ is increasing and $\delta \le \delta_1$, we have $\left(\tfrac{\delta_1}{3}-\tfrac{\delta}{3(1-\delta)}(1-\delta_1) \right) \ge \left(\tfrac{\delta_1}{3}-\tfrac{\delta_1}{3(1-\delta_1)}(1-\delta_1) \right) \ge 0$.
	Thus by ignoring $\left(\tfrac{\delta_1}{3}-\tfrac{\delta}{3(1-\delta)}(1-\delta_1) \right) n \le d^3 n/64$ edges of $M$, we get a subset $M' \subseteq M$ with $|M'|=(1-4 \delta_1/3) n$.
	
	Next, let $U' = U \setminus V(M')$ and $W'=W \setminus V(M')$ be the sets of vertices in $U$ and $W$, respectively, that are not incident to edges of $M'$.
	Note that both $U'$ and $W'$ have size $\delta_1 n/3$.
	We want to cover these vertices with triangles having the other two vertices in $V$.
	Any vertex $v \in U' \cup W'$ has degree at least $d n$ into $V$ and as $d > 2 \delta_0 \ge 2 \delta_1$ we can pick these triangles greedily for each $v \in U' \cup W'$ using $G_2$.
	Let $V' \subseteq V$ be the vertices that were used for these triangles and observe $|V \setminus V'|=|M'|=(1-4 \delta_1/3)n$.	
	
	To obtain the triangle factor it remains to find a perfect matching in $H_G(M',V \setminus V')$.
	By Lemma~\ref{lem:blow-up} it is sufficient to observe that the pair $(M',V \setminus V')$ is $(\eps',d^3/64)$-super-regular with respect to $H_G(M',V \setminus V')$, which holds because $(M,V)$ is $(\eps'/2,d^3/32)$-super-regular with respect to $H_G(M,V)$.
\end{proof}

Now we turn to the overview of the proof of Lemma~\ref{lem:tripartite}, for which $p \ge C \log n/n$.
We will rely again on Lemma~\ref{lem:trireg}, which gives a large matching $M \subseteq G(U,W,p)$ such that the pair $(M,V)$ is super-regular with respect to the auxiliary graph $H_G(M,V)$.
Starting from this matching $M$, we add more matching edges of $G(U,W,p)$ between the vertices not covered by $M$, and extend $M$ to a perfect matching in $G(U,W,p)$.
This will be possible using Lemma~\ref{lem:auxF} from below and Lemma~\ref{lem:matching}, where we emphasize that the $\log n$-term is essential for this last lemma.
It is then easy to find a perfect matching in $H_G(M,V)$ that gives a triangle factor.

Before we come to the proof of Lemma~\ref{lem:tripartite}, we introduce another auxiliary structure to describe in general which potential edges between $U$ and $W$ we would like to use for the matching $M$.
We define an auxiliary bipartite graph $F=F_{G,V}(U,W)$ with bipartition $U \cup W$, where a pair $(u,w) \in U \times W$ is an edge of $F$ if $u$ and $w$ have at least $d^2 n/2$ common neighbours in $G$, i.e.~$|N_G(u,V) \cap N_G(w,V)| \ge d^2 n/2$.
Similarly, for a set $X \subseteq V$, we call an edge $uw \in E(F)$ \emph{good for $X$} if there are at least $d^2|X|/2$ vertices $x \in X$ that are incident to $u$ and $w$ in $G$, i.e.~$uwx$ is a triangle in $G \cup F$.
We denote the spanning subgraph of $F$ with edges that are good for $X$ by $F_X$.
We prove the following lemma, which will be used in the proof of Lemma~\ref{lem:trireg} to construct $M$ and in the proof of Lemma~\ref{lem:tripartite} to extend $M$.

\begin{lemma}
	\label{lem:auxF}
	For any $\eps,d>0$ with $\eps \le d/2$, the following holds.
	Let $G$ be a graph on $U \cup V \cup W$, with $|V|=n$ and $(1-1/2)n \le |U|=n_0=|W| \le (1 + 1/2) n$, such that $(V,U)$ and $(V,W)$ are $(\eps,d)$-super-regular pairs with respect to $G$.
	Let $F=F_{G,V}(U,W)$ be the bipartite graph described above.
	Then $F$ satisfies the following properties.
	\begin{enumerate}[label=(\roman*)]
		\item \label{claim:min_degree} The minimum degree of $F$ is at least $(1-\eps)n_0$.
		\item \label{claim:degree_god} If $X \subset V$ and $|X| \ge 2\eps n/d$, then all but at most $\eps n_0$ vertices from $U$ have degree at least $(1-2\eps)n_0$ in $F_X$.
	\end{enumerate}
\end{lemma}

\begin{proof}[Proof of Lemma~\ref{lem:auxF}]
	Take any $u \in U$.
	Since $(U,V)$ is an $(\varepsilon,d)$-super-regular pair, we have $|N_G(u,V)| \ge d |V| \ge \varepsilon |V|$ and we can apply Lemma~\ref{lem:MDL} with $A=W$, $B=V$ and $Y=N_G(u,V)$ to conclude that the set
	\[  W_u  = \left\{w \in W: \deg_G(w,Y) \geq (d-\varepsilon) |Y| \right\}\]
	has size at least $(1-\eps)n_0$.
	Since $(d-\varepsilon) |Y| \geq (d-\varepsilon) d n \ge d^2n/2$, all the vertices in $W_u$ are neighbours of $u$ in the auxiliary graph $F$ and, therefore, $|N_F(u)| \ge (1-\varepsilon) n_0$.
	Analogously, we infer that $|N_F(w)| \ge (1-\eps) n_0$ for all $w \in W$ and hence $\delta(F) \ge (1-\eps)n_0$.
	
	If $|X| \ge 2 \eps n/d$, then by Lemma~\ref{lem:MDL} all but at most $\eps n_0$ vertices $u \in U$ have at least $d|X|/2$ neighbours in $X$.
	Fixing any $u \in U$ with this property, and again using Lemma~\ref{lem:MDL}, all but at most $\eps n_0$ vertices $w \in W$ have at least $d^2|X|/4$ common neighbours with $u$ in $X$.
	From~\ref{claim:min_degree} we know that $\delta(F) \ge (1-\eps)n_0$, therefore all but at most $\eps n_0$ vertices from $U$ have degree at least $(1-2\eps)n_0$ in $F_X$.
\end{proof}

Now we prove Lemma~\ref{lem:tripartite} using Lemmas~\ref{lem:trireg} and~\ref{lem:auxF}.

\begin{proof}[Proof of Lemma~\ref{lem:tripartite}]
	Given $d\in (0,1)$, let $\eps'>0$ be given by Lemma~\ref{lem:blow-up} on input $d^3/64$ and let $0<\delta<d^3/64$.
	Furthermore, let $0<\eps<\delta/4$ be given by Lemma~\ref{lem:trireg} on input $d$, $\delta$, and $\eps'/2$, let $C'$ be given by Lemma~\ref{lem:matching} for input $1/4$, set $C = 2 C'/\delta$, and let $p \ge C \log n/n$.
	
	Now let $G$ be a graph on $U \cup V \cup W$, with $|U|=|V|=|W|=n$, such that $(V,U)$ and $(V,W)$ are $(\eps,d)$-super-regular pairs with respect to $G$.
	We reveal the random edges in $G(U,W,p)$ in two rounds as $G_1 \sim G(U,W,p/2)$ and $G_2 \sim G(U,W,p/2)$.
	We apply Lemma~\ref{lem:trireg} with $G_1$ to a.a.s.~obtain a matching $M \subseteq G_1$ of size $|M| = (1-\delta)n$ such that the pair $(M,V)$ is $(\eps'/2,d^3/32)$-super-regular with respect to $H_G(M,V)$.
	Observe that $p \ge C \log n/n$, while the logarithmic factor is not required by Lemma~\ref{lem:trireg}, and thus the constant $C$ does not need to depend on the constant given by this lemma.
	
	Next, let $U' = U \setminus V(M)$ and $W'=W \setminus V(M)$ be the sets of vertices in $U$ and $W$ respectively that are not incident to edges of $M$. Note that both $U'$ and $W'$ have size $\delta n$.
	For the auxiliary graph $F=F_{G,V}(U,W)$ defined above, we consider the subgraph $F[U',W']$ induced by $U'$ and $W'$ and note that by Lemma~\ref{lem:auxF}~\ref{claim:min_degree} it has minimum degree at least $(\delta-\eps) n \ge \frac{3}{4} \delta n$.
	We use $G_2$ to reveal the edges of $F[U',W']$ with probability $p/2$ and a.a.s.~by Lemma~\ref{lem:matching} with $p/2 \ge C' \log (\delta n)/(\delta n)$ we find a perfect matching $M'$ in $G_2 \cap F[U',W']$.
	
	To obtain the triangle factor it remains to find a perfect matching in $H_G(M \cup M',V)$.
	For this we first greedily select a neighbour $v$ from $V$ for each $m \in M'$ and denote the set of vertices $v$ used in this way by $V'$.
	Since $|V'|=|M'|=\delta n$, it follows from the choice of $\delta$ that this is possible and that the pair $(M,V \setminus V')$ is $(\eps',d^3/64)$-super-regular with respect to $H_G(M,V \setminus V')$.
	As $|V \setminus V'|=|M|$, by Lemma~\ref{lem:blow-up}, there is a perfect matching in $H_G(M,V \setminus V')$ and we are done.
\end{proof}

It remains to prove Lemma~\ref{lem:trireg}.

\begin{proof}[Proof of Lemma~\ref{lem:trireg}]
	Given $0 < d,\delta,\eps' < 1$ with $2\delta \le d$, suppose that 
	\[\eps \ll \eta \ll \gamma \ll d,\delta,\eps'\]
	are positive real numbers such that
	\[
	\gamma \le \eps'/2, \quad
	\eta \log(1/\eta) \le \gamma/4, \quad
	\eps\le \eta d/4 \mbox{ and } \eps \le \gamma \delta/56\, .
	\]
	Furthermore, let $C>0$ such that $p \ge C/n$ is large enough for the applications of Chernoff's inequality below.
	
	Let $G$ be a graph on $U \cup V \cup W$, with $|V|=n$ and $|U|=|W|=n_0=(1 \pm 1/2)n$, such that $(V,U)$ and $(V,W)$ are $(\eps,d)$-super-regular pairs with respect to $G$.
	To find the matching $M$, construct a graph $\tilde{F} \subset F=F_{G,V}(U,W)$ by including each edge of $F$ randomly with probability $p$, independently from all other edges.
	
	\begin{claim}
		\label{claim:good_edges}
		A.a.s.~for every $X \subseteq V$ of size $\eta n$ and $U'\subseteq U$ and $W' \subseteq W$ both of size at least $\delta n_0$, the following statements hold.
		\begin{equation}
			\label{eq:edgesF1}
			\text{There are } (1 \pm \eps) e_F(U',W') p \text{ edges in }\tilde{F}[U',W']
		\end{equation}
		and
		\begin{equation}
			\label{eq:edgesF1good}
			\text{ at least  } (1-7\eps/\delta) |U'| |W'| p \text{ edges in } \tilde{F}[U',W'] \text{ are good for } X.
		\end{equation}
	\end{claim}
	
	\begin{claimproof}
		The expected number of edges between $U'$ and $W'$ in $\tilde{F}$ is by Lemma~\ref{lem:auxF}~\ref{claim:min_degree} at least $e_F(U',W') p \ge |U'|(|W'|-\eps n_0) p \ge C\delta^2 n/8$.
		Using Chernoff's inequality (Lemma~\ref{lem:chernoff}), we obtain with $C \delta^2 \eps^2 \ge 120$ that
		\[ \PP \left[ e_{\tilde{F}}(U',W') \neq (1 \pm \eps) e_F(U', W') p \right] \le 2 \exp \left( - \frac{\eps^2}{3} e_F(U',W') p \right) \le \exp \left(- 4n \right) .\]
		From Lemma~\ref{lem:auxF}~\ref{claim:degree_god} we get that the expected number of edges between $U'$ and $W'$ in $\tilde{F}$ that are good for $X$ is at least $(|U'|-\eps n_0)(|W'|-2\eps n_0) p \ge (1-3 \eps/\delta) |U'| |W'| p$.
		Moreover observe that
		\[
		\left(1-7\frac{\eps}{\delta} \right) = 
		\left(1-\frac{4\eps/\delta}{1-3\eps/\delta}\right) \left(1-3\eps/\delta\right) \le \left(1-\frac{4\eps/\delta}{\sqrt{1-3\eps/\delta}}\right) \left(1-3\eps/\delta\right)\, .
		\]
		Thus by Chernoff's inequality applied with $\tfrac{4\eps/\delta}{\sqrt{1-3\eps/\delta}}$ in place of $\delta$, we get with $C \eps^2 \ge 3$ that
		\[ \PP \left[ |E(\tilde{F}[U',W']) \cap E(F_X)| \le \left(1-7 \frac{\eps}{\delta} \right) |U'| |W'|p  \right] \le 2 \exp \left(- \frac{16 \eps^2}{3\delta^2} |U'| |W'|p \right) \le 2 \exp \left( -4n \right).\]
		Claim \ref{claim:good_edges} follows from the union bound over the at most $2^{3n}$ choices for $X$, $U'$, and $W'$. 
	\end{claimproof}
	
	Using a random greedy process we now choose a matching $M$ of size $(1-\delta)|W|$ between $U$ and $W$ in $\tilde{F}$ as follows.
	Having chosen edges $m_1,\dots,m_t\in \tilde{F}$ with $t<(1-\delta)|W|$, we pick $m_{t+1}$ uniformly at random from all edges of $\tilde{F}$ that do not share an endpoint with any of $m_1,\dots,m_t$.
	This is possible since, by~\eqref{eq:edgesF1} and Lemma~\ref{lem:auxF}~\ref{claim:min_degree}, between any two subsets of $U$ and $W$ of size $\delta |W|$ there is at least one edge of $\tilde{F}$.
	For $i=1,\dots,(1-\delta) n$ we denote by $\cH_i$ the history $m_1,\dots,m_i$.
	It remains to show that $(M,V)$ is $(\eps',d^3/32)$-super-regular with respect to the auxiliary graph $H=H_G(M,V)$.
	
	Observe, that any $m \in M$ has $|N_H(m)| \ge d^2n/2$ by construction.
	We now show that for any $v \in V$ we have $|N_H(v)| \ge d^3 n_0/32$.
	For this, it is sufficient to consider the first $dn_0/2$ edges $m_1, \dots, m_{dn_0/2}$ that are chosen for~$M$ by the random greedy process.
	For $i=1,\dots,dn_0/2$, by~\eqref{eq:edgesF1}, there are at most $e(\tilde{F}) \le (1+\eps)n_0^2p$  available edges to chose $m_i$ from.
	On the other hand, as long as $i \le dn_0/2$, the vertex $v$ has at least $d n_0/2$ neighbours in each of the sets $U$ and $W$ in $G$ that are not covered by the edges in  $\cH_{i-1}$. 
	Let $U' \subseteq U$ and $W' \subseteq W$ be the sets of those vertices and observe that $|U'|,|W'| \ge d n_0/2 \ge \delta n_0$.
	Therefore, by~\eqref{eq:edgesF1} and Lemma~\ref{lem:auxF}~\ref{claim:min_degree}, there are at least  
	\[(1-\eps) e_F(U',W') p \ge (1-\eps) (d/2) (d/2 -\eps) n_0^2 p\]
	edges in $\tilde{F}[(U \cup W) \setminus V(\cH_{i-1})]$ available to choose $m_i$ from, such that both endpoints of $m_i$ are incident to $v$ in $G$.
	
	Hence, for $i=1,\dots,dn_0/2$, we get
	\[\PP[m_i \in N_{H}(v) | \cH_{i-1}] \ge \frac{(1-\eps) (d/2) (d/2 -\eps) n_0^2 p}{(1+\eps) n_0^2 p} = \frac{(1-\eps) (d/2) (d/2 -\eps)}{(1+\eps)} \ge \frac{d^2}{8}.\]
	As this holds independently of the history of the process, this process dominates a binomial distribution with parameters $dn_0/2$ and $d^2/8$.
	Therefore, even though the events are not mutually independent, we can use Chernoff's inequality to infer that $|N_H(v)| \ge d^3n_0/32$ with probability at least $1-1/n^2$.
	Then, by applying the union bound over all $v \in V$, we obtain that indeed a.a.s.~$|N_H(v)| \ge d^3 n_0/32$ for all $v \in V$.
	
	Now we show the regularity of $H=H_G(M,V)$.
	Let $X \subseteq V$ be given with $|X| = \eta n$.
	For $i=0$, $1,\dots,t-1$, we obtain from~\eqref{eq:edgesF1} that there are at most $(1+\eps) (n_0-i)^2 p$ edges in $\tilde{F}[(U \cup W) \setminus V(\cH_i)]$ available for choosing $m_{i+1}$, of which, by~\eqref{eq:edgesF1good}, at least $(1-7\eps/\delta) (n_0 - i)^2 p$ are good for~$X$.
	Then
	\[ \PP \left[ m_{i+1} \text{ bad for }X \Big| \cH_{i-1} \right] \le 1-\frac{(1-7\eps/\delta) (n_0 - i)^2 p}{(1+\eps) (n_0-i)^2 p} = 1-\frac{1-7\eps/\delta}{1+\eps} \le 8 \eps/\delta\, .\]
	As the upper bound on the probability holds independently of the history, this process is dominated by a binomial distribution with parameters $(1-\delta)n_0$ and $8 \eps/\delta$.
	Observe that the expected value of this distribution is $(1-\delta) n_0 \,  8\eps/\delta \le 8 \eps n_0/\delta$.
	Then, with $B_X \subseteq M$ being the edges in $M$ that are not good for $X$ and as $\gamma n_0 \ge 7 \cdot 8 \eps n_0/\delta$, we get from Chernoff's inequality (second part of Lemma~\ref{lem:chernoff}) that
	\[ \mathbb{P}\left[ |B_X|  > \gamma n_0 \right] \le  \exp \left(-\gamma n_0 \right)\,.\]
	
	There are at most $\binom{n}{\eta n} \le \left(e/\eta \right)^{\eta n} \le  \exp (\eta \log(1/\eta) n) \le \exp(\gamma n_0/2)$ choices for $X$ and, thus, with the union bound over all these choices, we obtain that a.a.s.~there are at most $\gamma n_0$ bad edges in $M$ for any $X \subseteq V$ with $|X|=\eta n$.
	
	Therefore for any set $X' \subseteq V$ and $M'\subseteq M$ with $|X'| \ge \eps' n$ and $|M'|\ge \eps' |M|$  we find
	\[ e_H(M',X') \ge (|M'| - \gamma n_0) \frac{d^2 \eta n}{2} \frac{|X'|}{2\eta n} \ge \frac{d^2}{8} |M'| |X'|.\]
	Overall, we can conclude the pair $(M,V)$ is $(\eps',d^3/32)$-super-regular with respect to the auxiliary graph $H_G(M,V)$.
\end{proof}

\subsection{Proof of Lemma~\ref{lem:bipartite}.}
Lemma~\ref{lem:bipartite} easily follows from Lemma~\ref{lem:tripartite2} after splitting the sets $U$ and $V$ appropriately.

\begin{proof}[Proof of Lemma~\ref{lem:bipartite}]
	Let $0<d<1$, choose $\delta'$ with $0 < \delta' \le d/4$ and apply Lemma~\ref{lem:tripartite2} with $\delta'$ and $d/4$ to obtain $\delta_0,\delta,\eps'$ with $\delta' \ge \delta_0 > \delta>\eps'>0$ and $C'>0$.
	Then let $\eps$ with $0 < \eps \le \eps'/8$ and $C \ge 10 C'$.
	Next let $U$ and $V$ be vertex sets of size $|V|=n$ and $3n/4 \le |U|=n_0 \le n$ where $n+n_0 \equiv 0 \pmod 3$ and assume that $(U,V)$ is a $(\eps,d)$-super-regular pair.
	
	The idea is to split the regular pair $(U,V)$ into two regular cherries and then use Lemma~\ref{lem:tripartite2} in each cherry.
	More precisely, we partition $V$ into $V_1$, $U_2$, $W_2$ and $U$ into $V_2$, $U_1$, $W_1$ such that for $i=1,2$ the pairs $(U_i,V_i)$ and $(W_i,V_i)$ are $(\eps',d/4)$-super-regular pairs and $(1-\delta_0)|V_i| \le |U_i|=|W_i| \le (1-\delta) |V_i|$.
	To obtain this we split the sets according to the following random process.
	We put any vertex of $V$ into each of $U_2$ and $W_2$ with probability $q_1$ and into $V_1$ with probability $1-2q_1$.
	Similarly, we put any vertex of $U$ into each of $U_1$ and $W_1$ with probability $q_2$ and into $V_2$ with probability $1-2q_2$.
	We choose $q_1$ and $q_2$ such that the expected sizes satisfy for $i=1,2$
	\[ \EE[|U_i|]=\EE[|W_i|] = \left(1-\frac{\delta_0+\delta}{2}\right) \EE[|V_i|]. \]
	This is possible since such conditions give a linear system of two equations in two variables $q_1$ and $q_2$.
	As $3n/4 \le n_0 \le n$, the solution satisfies $1/7 < q_1,q_2 < 3/7$.
	Then by Chernoff's inequality (Lemma~\ref{lem:chernoff}) and with $n$ large enough there exists a partition such that for $i=1,2$ we have that $|W_i|$, $|U_i|$, and $|V_i|$ are all within $\pm n^{2/3}$ of their expectation 
	and the minimum degree within both pairs $(U_i,V_i)$ and $(W_i,V_i)$ is at least a $d/2$-fraction of the other set.
	For $i=1,2$ we redistribute $o(n)$ vertices between $U_i$ and $W_i$ and move at most one vertex from or to $V_i$ to obtain
	\[ (1-\delta_0)|V_i| \le |U_i| = |W_i| \le (1-\delta)|V_i| \]
	with minimum degree within both pairs $(U_i,V_i)$ and $(W_i,V_i)$ at least a $d/4$-fraction of the other set.
	
	From this we get that for $i=1,2$ the pairs  $(U_i,V_i)$ and $(W_i,V_i)$ are $(\eps',d/4)$-super-regular.
	Also note that $\min \{ |V_1|, |V_2|\} \ge \min \{ (1-2.1q_1)n , (1-2.1q_2)n_0 \} \ge n/10$ by the bound on $q_1$ and $q_2$ from above.
	Then for $i=1,2$ and with $C/n \ge C'/ \min\{ |V_1|, |V_2|\}$ we a.a.s.~get a triangle factor on $V_i$, $U_i$, $W_i$ from Lemma~\ref{lem:tripartite2}. 
	Together these give a triangle factor on $U \cup V$.
\end{proof}

\section{Concluding remarks and open problems}
\label{sec:concluding}

In this paper we close the last gap for the existence of a triangle factor in randomly perturbed graphs $G_\alpha \cup G(n,p)$ and now the threshold is determined for any $\alpha \in [0,1]$ (c.f.~Table~\ref{fig:K3-factor}).
In this last section we would like to discuss a possible improvement for Theorem~\ref{thm:sublinear} and three different directions for generalisations of our results.

\subsection{Optimality of the Sublinear Theorem}
The bound on $p$ in Theorem~\ref{thm:zero} is asymptotically optimal.
However, as discussed earlier, when $m$ is significantly smaller than $n/3$, then $K_{m,n-m} \cup G(n,p)$ contains $m$ pairwise vertex-disjoint triangles already at $p \ge C/n$.
When dealing with the sublinear case $1 \le m \le n/256$ in Theorem~\ref{thm:sublinear}, we require $p \ge C \log n/n$; in fact we do not use the $\log n$-term when $(\log n)^3 \le m \le \sqrt{n}$ (notice indeed that Proposition~\ref{prop:lesqrt} is stated with $p \ge C/n$).
It would be interesting to understand if $p \ge C/n$ also suffices for larger values of $m$ and if Proposition~\ref{prop:gesqrt} can be improved. 
\begin{problem}
	Show that there exists $C >0$ such that for any $\sqrt{n} \le m \le n/32$ and any $n$-vertex graph $G$ with maximum degree $\Delta(G) \le n/32$ and minimum degree $\delta(G)\ge m$, a.a.s.~$G \cup G(n,p)$ contains $m$ pairwise vertex-disjoint triangles, provided $p \ge C /n$. 
\end{problem}
Note that this would improve Theorem~\ref{thm:extremal} when $\alpha<1/3$.
Indeed in our proof we only need the $\log n$-term as we apply Theorem~\ref{thm:sublinear}, while we could use Lemma~\ref{lem:tripartite2} instead of~\ref{lem:tripartite} and avoid the $\log n$-term in the rest of the argument.

\subsection{Larger cliques}
Similarly to a triangle factor, for any integer $r \ge 2$ and any $\alpha \in [0,1]$, one can look at the threshold for the existence of a $K_r$-factor in $G_\alpha \cup G(n,p)$.
The result of Johannson, Kahn and Vu~\cite{johansson2008factors} determines the threshold in $G(n,p)$ at $n^{-2/r} \log^{2/(r^2-r)} n$ and for $\alpha \ge 1-1/r$ the existence in $G_\alpha$ alone is proved by Hajnal-Szemer\'edi Theorem~\cite{hajnal1970proof}.
For other values of $\alpha$, the results of Balogh, Treglown, and Wagner~\cite{balogh2019tilings}, and of Han, Morris and Treglown~\cite{han2019tilings} can be summarised in the following theorem generalising Theorem~\ref{thm:k3_factor}.

\begin{theorem}
	\label{thm:kr_factor}
	For any integers $2 \le k \le r$ and any $\alpha \in (1-\frac{k}{r} , 1-\frac{k-1}{r})$ there exists $C>0$ such that the following holds. 
	For any $n$-vertex graph $G$ with minimum degree $\delta(G) \ge \alpha n$ there a.a.s.~is a $K_r$-factor in~$G \cup G(n,p)$ provided that $p \ge C n^{-2/k}$.
\end{theorem}

However, this leaves open the question for the threshold in the boundary cases, i.e. when $\alpha=1-k/r$ for $2 \le k \le r-1$.
A natural extremal structure is given by the complete $\lceil r/k \rceil$-partite graph with $\lfloor r/k \rfloor$ classes of size $kn/r$ and possibly one class of size $(1-\lfloor r/k \rfloor k/r)n$ if $k \nmid r$.
This implies that to get a $K_r$-factor in $G_\alpha \cup G(n,p)$ we need to cover all but $\operatorname{polylog} n$ vertices of the sets of size $kn/r$ with vertex-disjoint copies of $K_k$, i.e.~the threshold is at least $n^{-2/k} (\log n)^{2/\left(k^2-k\right)}$.
Surprisingly, this is not sufficient in the case when $r>3$ and $k \ne 2$; in fact, for small $\eps$, even $n^{-2/k+\eps}$ is not sufficient.

We briefly explain the counterexample for $r=4$ and $k=3$, by constructing an $n$-vertex graph $G$ with minimum degree $\delta(G) \ge (1-3/4)n = n/4$ such that even for small $\eps>0$ and $p \ge n^{-2/3+\eps}$, a.a.s.~the graph $G \cup G(n,p)$ does not contain a $K_4$-factor.
Let $0<\eps \le 1/49$, $p \ge n^{-2/3+\eps}$, and $n^{7 \eps} \le m \le n^{1/7}$.
Then, for two sets $A$ and $B$ with $|A|=n/4-m$ and $|B|=3n/4+m$, we let $G$ be the $n$-vertex graph on $V(G)=A \cup B$ such that $A$ is an independent set, $G[B]$ is given by $|B|/(2m)$ disjoint copies of $K_{m,m}$, and any pair of vertices $(a,b)$ with $a \in A$ and $b \in B$ is an edge.
Clearly $G$ has minimum degree $n/4$.
If $G \cup G(n,p)$ contains a $K_4$-factor, since $A$ only contains $n/4-m$ vertices, at least $m$ copies of $K_4$ must lie within $B$.
However we claim that a.a.s.~the perturbed graph $G \cup G(n,p)[B]$ contains less than $m$ copies of $K_4$ and thus a.a.s.~$G \cup G(n,p)$ does not contain a $K_4$-factor.
Denote by $X$ the number of $K_4$'s in $G \cup G(n,p)[B]$.
It is not hard to see that when $m$ is not too small the best way to build a $K_4$ in $B$ is to choose a $K_{m,m}$ in $B$ and ask for an edge of $G(n,p)$ on each side of $K_{m,m}$.
We get $\mathbb{E}[X] \lesssim \frac{n}{m} m^4 p^2 = o(m)$ and by Markov's inequality a.a.s.~$X < m$ as claimed.
\begin{problem}
	Determine the behaviour of the threshold and the extremal graphs at and around minimum degree $n/4$.
\end{problem}

The counterexamples for other values of $r>3$ and $k\ne 2$ can be constructed in a similar way, by slightly modifying the corresponding extremal graph defined above.
In the case when $k=2$ this construction does not increase the lower bound $n^{-1} \log n$ and, with Theorem~\ref{thm:sublinear} in mind, we believe that Theorem~\ref{thm:main} generalises to $K_r$.
However we believe that in all cases, using our methods, Theorem~\ref{thm:non-extremal} can be extended to $K_r$: i.e.~for all $2 \le k \le r-1$ and with $\alpha=1-k/r$, when $G_\alpha$ is not close (with a similar condition as in Definition~\ref{def:stability}) to the extremal graph defined above, then $p \ge C n^{-2/k}$ is sufficient for a $K_r$-factor in $G_\alpha \cup G(n,p)$.

\subsection{Longer cycles}
Another possible generalisation is to consider factors of longer cycles $C_{\ell}$.
With slight adjustments to our methods (mainly to Theorem~\ref{thm:sublinear}), we are able to show the following theorem.

\begin{theorem}
	\label{thm:cycles}
	For any integer $\ell \ge 3$ there exists $C>0$ such that for any $n$-vertex graph $G$ we can a.a.s.~find at least $\min\{ \delta(G), \lfloor n/\ell \rfloor \}$ pairwise vertex-disjoint $C_\ell$'s in $G \cup G(n,p)$, provided that $p\ge C \log n /n$. 
\end{theorem}

We will discuss this in~\cite{BPSS_cycles_paper,BPSS_cycles}, including the corresponding variants for the finer statements, Theorems~\ref{thm:non-extremal}--\ref{thm:sublinear}.

\subsection{Universality}

Even further, we call a graph $2$-universal if it contains any $n$-vertex graph of maximum degree $2$ as a subgraph.
It is known that the threshold for $2$-universality is asymptotically the same as for a triangle factor when $\alpha <1/3$ or $\alpha \ge 2/3$~\cite{aigner1993embedding,ferber2019optimal,parczyk20202}.
In a follow up paper we will expand our approach and prove that this also holds for the remaining cases, i.e.~$\log n/n$ gives the threshold for $\alpha=1/3$ and $1/n$ gives the threshold when $1/3<\alpha<2/3$.

\section*{Acknowledgements}
We thank an anonymous referee for their insightful and valuable comments.

\bibliographystyle{abbrv}
\bibliography{references}

\appendix

\section{Proof of Lemmas~\ref{lem:triangle},~\ref{lem:matching}, and~\ref{lem:stable_cluster}}
\label{sec:appendix}	
For the proof of Lemma~\ref{lem:triangle} we use Janson's inequality (see e.g.~\cite[Theorem 2.14]{JLR}).
\begin{lemma}[Janson's inequality]
	\label{lem:janson}
	Let $p \in (0,1)$ and consider a family $\{ H_i \}_{i \in \cI}$ of subgraphs of the complete graph on the vertex set $[n]=\{1,\ldots,n\}$. For each $i \in \cI$, let $X_i$ denote the indicator random variable for the event that $H_i \subseteq \Gnp$ and, write $H_i \sim H_j$ for each ordered pair $(i,j) \in \cI \times \cI$ with $i \neq j$ if $E(H_i) \cap E(H_j) \not= \emptyset$.
	Then, for $X = \sum_{i \in \cI} X_i$, $\mathbb{E}[X] = \sum_{i \in \cI} p^{e(H_i)}$,
	\begin{align*}
		\delta = \sum_{H_i \sim H_j} \mathbb{E}[X_i X_j] = \sum_{H_i \sim H_j} p^{e(H_i) + e(H_j) - e(H_i \cap H_j)}
	\end{align*}
	and any $0 < \gamma < 1$ we have
	\begin{align*}
		\mathbb{P} [X \le (1-\gamma) \mathbb{E}[X]] \le \exp \left(-\frac{\gamma^2 \mathbb{E}[X]^2}{2(\mathbb{E}[X] + \delta)} \right).
	\end{align*}
\end{lemma}

\begin{proof}[Proof of Lemma~\ref{lem:triangle}]
	Let $d>0$ and pick $C > 68/d^3$ and $p>C/n$.
	Let $\cI = E_G(U,W) \times V$ and for each $i=(uw,v) \in \cI$ let $H_i$ be the path $uvw$ on the three vertices $u \in U$, $v \in V$ and $w \in W$.
	We want to apply Lemma~\ref{lem:janson} to the family $\{H_i\}_{i \in \cI}$.
	Using the same notation, we have $\EE[X] = |\cI| p^2 \ge d n^3 p^2$ and $\delta \le n^2 (2n) n p^3$, as for $i=(uw,v)$ and $j=(u'w',v')$ with $i, j \in \cI$ and $i \neq j$ we have $H_i \sim H_j$ if and only if $v=v'$ and precisely one of the equalities $u=u'$ and $w=w'$ holds.
	Then the Janson's inequality with $\gamma=1/2$ gives 
	\[
	\mathbb{P} \left[X \le  \frac{\EE[X]}{2}\right] \le \exp \left(-\frac{\gamma^2}{2} \frac{dn^3p^2}{1+2np/d} \right) \le \exp \left( - \frac{1}{17}d^2n^2p \right) \le 2^{-4n/d},
	\]
	as $p>C/n$ and $C > 68/d^3$.
	Thus with probability at least $1-2^{-4n/d}$ we have $X > \EE[X]/2$ and there is at least one path $H_i$ on vertices $u,v,w$ for some $i=(uw,v) \in \cI$ in $G(n,p)$.
	As $uw$ is an edge of $G$ by definition of $\cI$, we get a triangle in $G \cup G(U \cup W,V,p)$ with one vertex in each of $U,V,W$, as required.
\end{proof}

Next, we prove that if a bipartite graph $G$ has large minimum degree, its random subgraph $G_p$ contains a perfect matching if $p \ge C \log n/n$.

\begin{proof}[Proof of Lemma~\ref{lem:matching}]
	For $\eps>0$ let $C \ge 3/\eps$.
	Let $G$ be a bipartite graph with partition classes $U$ and $W$ of size $n$ and minimum degree at least $(1/2+\eps)n$.
	Let $p \ge C \log n/n$ and suppose that $G_p$ does not have a perfect matching. 
	Then, by Hall's Theorem, there exists a set $S \subseteq U$ or $S \subseteq W$ with $|S| > |N(S)|$. 
	Let $S$ be a minimal such set, then $|S|=|N(S)|+1$, $|S| \leq \lceil n/2 \rceil$, and every vertex in $|N(S)|$ is adjacent to at least two vertices of $S$.
	
	Let $1 \leq s \leq \lceil n/2 \rceil$ and let $A_s$ denote the event that there is a minimal set $S$ of size $s$ violating Hall's condition. 
	If $s=1$, then $S$ is an isolated vertex. 
	Let $X_n$ be the number of isolated vertices in $G_p$, then for $n$ large enough
	\[ 
	\EE[X_n] \le 2n \left(1-\frac{C \log n}{n}\right)^{(1/2+\eps) n} \le 2n \exp(-C (1/2+\eps) \log n) \le 2 n^{1-C(1/2+\eps)}
	\]
	and $\lim_{n \rightarrow \infty} \EE[X_n] = 0$. 
	Therefore, a.a.s.~$G_p$ does not contain isolated vertices.
	
	If $s \geq 2$, then
	\begin{align*}
		\PP[A_s] & \leq \binom{n}{s} \binom{n}{s-1} \left( \binom{s}{2} p^2 \right)^{s-1} \left( 1-p \right)^{s \left( (1/2+\varepsilon) n - s + 1 \right)}  \leq \left( \frac{en}{s} \right)^s \left( \frac{en}{s-1} \right)^{s-1} s^{2(s-1)} (1-p)^{s \varepsilon n} \\
		& \leq e^{2s-1} \frac{s^{2(s-1)}}{s^s (s-1)^{s-1}} n^{2s-1} \exp(-s \varepsilon C \log n) \leq e^{2s-1} n^{-s(\varepsilon C - 2) -1} = \frac{1}{e n} \left( \frac{e^2}{n^{\varepsilon C -2}} \right)^s.
	\end{align*}
	
	We have
	\[\PP[A_s] \leq \frac{1}{e n} \left( \frac{e^2}{n} \right)^s.\]
	The probability that there is a minimal set $S$ violating Hall's condition is then bounded by
	\[ \sum_{s=2}^{\lceil n/2 \rceil} \PP[A_s] 
	\leq \sum_{s=2}^{\lceil n/2 \rceil} \frac{1}{e n} \left( \frac{e^2}{n} \right)^s 
	\leq \frac{1}{e n} \sum_{s \geq 0} \left( \frac{e^2}{n} \right)^s
	= \frac{1}{e n} \frac{1}{1-e^2/n}.\]
	As $\lim_{n \rightarrow \infty} \frac{1}{e n} \frac{1}{1-e^2/n} = 0$, a.a.s.~$G_p$ does not violate Hall's condition and, therefore, contains a perfect matching.
\end{proof}

Finally we prove Lemma~\ref{lem:stable_cluster} following the argument of~\cite[Lemma $12$ and $13$]{BMS_cover}.
For this we consider a largest matching $M$ in the reduced graph $R$ and assume that $|M|<(\alpha+2d)t$.
Then we will find a set $I \subset V(R)$ of size roughly $(1-\alpha) t$ which contains very few edges.
With the properties of the reduced graph, we conclude that the original graph $G$ is $(\alpha,\beta)$-stable.

\begin{proof}[Proof of Lemma~\ref{lem:stable_cluster}]
	Given $0<\beta<\tfrac{1}{12}$ let $0<d<10^{-4} \beta^6$.
	Then let $0<\eps<d/4$, $4 \beta \le \alpha \le \tfrac 13$, and $t \ge \tfrac{10}{d}$.
	Next let $G$ be an $n$-vertex graph on vertex set $V$ with minimum degree $\delta(G) \ge (\alpha -\tfrac12 d)n$ that is not $(\alpha,\beta)$-stable and let $R$ be the $(\eps,d)$-reduced graph for some $(\eps,d)$-regular partition $V_0,\dots,V_t$ of $G$.
	We observe for the minimum degree of $R$ that $\delta(R) \ge (\alpha-2d) t$ because, otherwise, there would be vertices with degree at most $(\alpha-2d)t(n/t)+\eps n < (\alpha - d/2)n - (d+\eps)n$ in $G'$ contradicting~\ref{prop:degree}.
	
	Let $M$ be a matching in $R$ of maximal size.
	Observe that $|M| \ge \min \{ \delta(R),\lfloor t/2 \rfloor \} \ge (\alpha-2d)t$. 
	We assume $|M| < (\alpha+2d)t$ and show that $G$ must then be $(\alpha,\beta)$-stable, which is a contradiction.
	Let $U = V(R) \setminus V(M)$.
	We shall first show that there exists a set $I \subset V(R)$ of size $|U|+|M|$ that contains only few egdes.
	
	Since $M$ is a matching of maximal size in $R$, $U$ is independent.
	Moreover, given an edge $xy \in M$, either $x$ or $y$ has at most one neighbour in $U$.
	Then we can split $V(M)$ into two disjoint subsets $X$ and $Y$ by placing for each matching edge $xy$ of $M$ one of its endpoints with at most one neighbour in $U$ into the subset $X$, and the other endpoint into the subset $Y$.
	We claim that $I= U \cup X$ contains only few edges.
	We have $e(U)=0$, $e(X,U) \le |X|$, and we can upper bound $e(X)$ as follows.
	Let $xy \in E(X)$ and denote by $x'$ and $y'$ the vertices matched to $x$ and $y$ in $M$ respectively.
	Then $x', y' \in Y$ and either $x'$ or $y'$ has at most one neighbour in $U$.
	Otherwise, there would be two distinct vertices $x'', y'' \in U$ such that $x'x''$ and $y'y''$ are edges of $R$, and we could apply the rotation $M \setminus \{xx',yy'\} \cup \{xy, x'x'',y'y''\}$ and get a larger matching, contradicting the maximality of $M$.
	Therefore, $e(X) \le |X||Z|$, where $Z=\{v \in Y| \deg(v,U) <2 \}$.
	Observe that 
	\[e(Y,U) \le  (|Y| - |Z|) |U| + |Z|\]
	and 
	\[  e(Y,U) \ge |U| \delta(R) - e(X,U) \ge |U| \delta(R) - |X|\]
	where we use that since $U$ is independent, a vertex in $U$ can have neighbours only in $X$ and $Y$.
	We get
	\[ |Z| \le  \frac{(|Y| -\delta(R)) |U|+|X|}{|U|-1} < \frac{4dt |U| + |X|}{|U|-1} \le  5dt, \]
	where the first inequality comes from the upper and lower bound on $e(Y,U)$, the second one from $|Y| = |M| < (\alpha+2d)t$ and $\delta(R) \ge (\alpha-2d)t$, and the last one from $|U|=t-2|M| \ge t/4$, $|X| = |M| < t/2$ and $10/t \le d$.
	Hence, $e(X) \le |X| 5dt$.
	
	Therefore, the set $I=U \cup X$ has size
	\[ |I| = |V(R)| - |M| = \left(1-\alpha \pm 2d\right)t\]
	because $|Y|=|M| = (\alpha \pm 2d)t$
	and contains at most
	\[e(I) \le e(X) + e(X,U)+ e(U) \le |X|5dt+|X| \le (5 dt+1) (\alpha +2d) t \le 6 \alpha dt^2\]
	edges, where we use $|X|=|M| = (\alpha \pm 2d)t$, $d \le \alpha/20$ and $10/t \le d$ in the last inequality.
	
	We now move to the original graph $G$ and prove that the existence of such set $I$ in $R$ implies that $G$ is $(\alpha,\beta)$-stable.
	Let $B''=\bigcup_{i \in I} V_i$ be the union of the clusters $I$.
	Then $|B''| = (1-\alpha \pm 3d)n$ and $e(B'') \le 6 \alpha dn^2$.
	Let 
	\[B'=\{v \in B''| \deg(v,B'') \le \sqrt{d} n\}.\]
	Then $e(B'') \ge (|B''|-|B'|) \sqrt{d} n$ and, therefore, all but at most $6 \alpha \sqrt{d}n $ vertices of $B''$ belong to $B'$ and, thus, $|B'| = (1-\alpha \pm 4\sqrt{d})n$.
	Let
	\[A'=\{v \in V | \deg(v,B') \ge (1-\beta/4)|B'|\}\]
	and note that $A'\cap B'=\emptyset$.
	Observe that if $v \in B'$, then	
	\begin{equation}
		\label{eq:B'}
		\deg(v,V \setminus B') \ge \delta(G) - \deg(v,B') \ge (\alpha -d/2 -\sqrt{d})n \ge (\alpha - 2\sqrt{d})n.
	\end{equation}
	With $|V \setminus B'| \le (\alpha + 4 \sqrt{d})n$ this implies 
	\[ e(B',V \setminus B') \ge (\alpha -2 \sqrt{d}) n |B'| \ge (|V \setminus B'|-6 \sqrt{d}n) |B'| \]
	and with the definition of $A'$ and the fact that $|V \setminus (B' \cup A')| = |V \setminus B'| - |A'|$, we get
	\[
	e(B',V \setminus B') 
	\le |V \setminus (B' \cup A')| (1-\beta/4) |B'| + |A'| |B'| 
	= \big( |V \setminus B'| - |V \setminus (B' \cup A')| \beta /4 \big) |B'|\, .
	\]
	The last two inequalities imply that all but at most $24 \sqrt{d}n/\beta $ vertices of $V \setminus B'$ belong to $A'$.
	Therefore we can bound the size of $A'$ as follows
	\[
	|A'| \ge |V \setminus B'| - 24 \sqrt{d}n/\beta \ge \alpha n - 4 \sqrt{d}n  - 24 \sqrt{d}n/\beta \ge \alpha n - \beta^2 n\, ,
	\]
	and
	\[
	|A'| \le |V \setminus B'| \le \alpha n + 4 \sqrt{d}n \le \alpha n + \beta^2 n\, .
	\]
	where we used in both inequalities that $4 \sqrt{d}n + 24 \sqrt{d}n/\beta \le \beta^2 n$, as $d \le 10^{-4}  \beta^6$.
	
	It follows that we have built two sets $A'$ and $B'$ such that $|A' \cup B'| \ge n-\beta^2 n$, $|A'| = \alpha n \pm \beta^2 n$ and $|B'| = (1-\alpha)n \pm \beta^2 n$.
	Moreover each vertex of $A'$ has at least $(1-\beta/4)|B'|$ neighbours in $B'$ by the definition of $A'$, and each vertex of $B'$ has at least $(1-\beta/2)|A'|$ neighbours in $A'$.
	This can be justified as follows.
	Given $v \in B'$,
	\begin{align*}
		\deg(v,A')&=\deg(v,V \setminus B') - \deg(v, V \setminus (A' \cup B')) \\
		& \ge \deg(v,V \setminus B') - |V \setminus (A' \cup B')| \\
		& \ge (\alpha - 2 \sqrt {d})n - 24\sqrt{d}n/\beta \ge (\alpha-\beta^2)n\\
		& \ge \frac{\alpha - \beta^2}{\alpha + \beta^2}|A'| \ge ( 1-\beta/2 ) |A'|\, ,
	\end{align*}
	where we used that $A'$ and $B'$ are disjoint, the inequalities~\eqref{eq:B'}, $|V \setminus (A' \cup B')| \le 24 \sqrt{d}n/\beta $ and $2 \sqrt {d}+ 24\sqrt{d}/\beta \le \beta^2$, the upper bound on $|A'|$ and the inequality $\alpha \ge 4 \beta$.
	
	Now we need to take care of the vertices of $G$ not yet covered by $A' \cup B'$, i.e.~the at most $\beta^2 n$ vertices in $V \setminus (A' \cup B')$.
	Let $v$ be one such vertex. 
	Then $\deg(v,A' \cup B') \ge \delta(G) - |V \setminus (A' \cup B')| \ge (\alpha - d/2)n -\beta^2 n \ge \alpha n/2$. 
	Therefore, it is possible to add these vertices to $A'$ and $B'$ to obtain $A \supseteq A'$ and $B \supseteq A'$ such that each vertex of $B$ has at least $\alpha n/4$ neighbours in $A$, and each vertex of $A$ has at least $\alpha n/4$ neighbours in $B$.
	As we add at most $\beta^2 n$ vertices, we have $|A| = (\alpha \pm 2 \beta^2)  n$ and $|B| = (1-\alpha\pm 2\beta^2)n$.
	Moreover, all but at most $\beta^2 n \le \beta n$ vertices from $A$ have degree at least 
	\[
	(1-\beta/4)|B'|\ge (1-\beta/4)(|B|-\beta^2 n) \ge (1-\beta)|B|
	\]
	into $B$, where we used that $|B| \le |B'|+\beta^2 n$, $|B| \ge (1-\alpha-2\beta^2)n \ge (2/3-2\beta^2)n$ and $\beta < 1/12$.
	Similarly, all but at most $\beta^2 n \le \beta n$ vertices from $B$ have degree at least $(1-\beta/2)|A'|\ge (1-\beta)|A|$ into $A$.
	Moreover, as $B'$ is a subset of $B''$ and we add at most $\beta^2 n$ vertices to $B'$ to get $B$, we have $e(B) \le e(B'') + \beta^2 n^2 \le (6 \alpha d + \beta^2)n^2 \le \beta n^2$. 
	Therefore, $G$ is $(\alpha,\beta)$-stable according to Definition~\ref{def:stability}.
\end{proof}

\end{document}